\documentclass{amsart}
\usepackage{wrapfig}
\usepackage{graphicx} 
\usepackage{multicol} 

\usepackage{tikz}
\usepackage{mathrsfs}

\usepackage[latin1]{inputenc}
\usepackage[T1]{fontenc}
\usepackage{verbatim}
\usepackage{multirow}
\usepackage{geometry}
\usepackage{amssymb}
\usepackage{amsmath}
\usepackage{graphicx}
\usepackage{amsthm}
\usepackage{caption}
\usepackage{color}
\usepackage{enumerate}
\usepackage{graphicx}
\usepackage{float}
\usepackage{subfigure}

\newcommand{\eChar}{\begin{enumerate}[(i)]}
\newcommand{\eCharR}{\begin{enumerate}[(a)]}
\newcommand{\eBr}{\begin{enumerate}[(1)]}

\newcommand{\Abstract}

\title
{
Edge-connectivity and non-negative Lin--Lu--Yau curvature
}

\author{Shiping Liu}
\address{
School of Mathematical Sciences\\
University of Science and Technology of China\\
96 Jinzhai Road\\
Hefei 230026\\
Anhui Province\\
China}
\email{spliu@ustc.edu.cn}

\author{Qing Xia}
\address{
School of Mathematical Sciences\\
University of Science and Technology of China\\
96 Jinzhai Road\\
Hefei 230026\\
Anhui Province\\
China}
\email{xq0420@mail.ustc.edu.cn}

\date{\today}
\theoremstyle{plain}
\newtheorem{lemma}{Lemma}[section]
\newtheorem{theorem}[lemma]{Theorem}

\newtheorem{corollary}[lemma]{Corollary}
\theoremstyle{definition}

\newtheorem{claim}[lemma]{Claim}

\newtheorem{definition}[lemma]{Definition}
\newtheorem{remark}[lemma]{Remark}

\numberwithin{equation}{section}

\begin{document}

\pagestyle{plain}

\begin{abstract}We determine the edge-connectivity of all connected graphs with non-negative Lin--Lu--Yau curvature. As a consequence, we prove that the edge-connectivity of a \emph{finite} connected graph with non-negative Lin--Lu--Yau curvature is equal to its minimum degree. This answers an open question of Chen, the first named author and You. We actually classify all connected graphs with non-negative Lin--Lu--Yau curvature and edge-connectivity smaller than their minimum degree. They are all infinite graphs with two ends.
\end{abstract}
\maketitle


\section{Introduction and statement of results}

The interplay between local and global properties of spaces has long been a central theme in the study of both geometry and graph theory. The local properties of spaces are often described by curvature bounds in geometry \cite{Jost} and by local graphs in graph theory \cite{Hall80, Weetman94a, Weetman94b}. With various synthetic notions of discrete curvature, the approaches of the two disciplines have interacted quite deeply \cite{NR17}.  In this paper, we explore connections between local properties as captured by the Lin--Lu--Yau curvature and global properties---edge-connectivity and finiteness---of locally finite graphs. 

The edge-connectivity $k'(G)$ of a locally finite graph $G$ with at least two vertices is the minimum cardinality of an edge cut of $G$. Here, an edge cut is a set of edges whose deletion increases the number of connected components. We call an edge cut with the minimum cardinality a min-cut of $G$. The edge-connectivity of a single vertex is defined to be 0. 
(Notice that the notation $k(G)$ is reserved for the vertex-connectivity of a graph $G$.) By definition, we have
\[k'(G) \leq \delta(G),\]
where $\delta(G)$ stands for the minimum vertex degree of $G$. This estimate is sharp. In fact, it has been shown that the equality holds for any finite connected edge-transitive graphs \cite{Mad71,Wat70}, vertex-transitive graphs \cite{Mad71}, and distance-regular graphs \cite{BrHa05, BrMe85}. By a recent result \cite{CKL25}, the equality still holds for any (possibly infinite) connected amply regular graphs, for which any two vertices at distance $2$ have more than one common neighbor. On the other hand, the gap between the edge-connectivity and minimum vertex degree can be arbitrarily large. Indeed, for any integers $0<\ell\leq d$, there exists a graph $G$ with $k'(G)=\ell$ and $\delta(G)=d$.

The Lin--Lu--Yau curvature is a discrete analogue of the Ricci curvature in Riemannian geometry. It was proposed by Lin, Lu and Yau \cite{LLY11} as a modification of Ollivier's notion of coarse Ricci curvature \cite{Ollivier09}. 
The Lin--Lu--Yau curvature of an edge $x\sim y\in E$ is defined via comparing the two neighborhoods around $x$ and $y$ in terms of Wasserstein distance. Intuitively, the curvature along an edge $x\sim y\in E$ is positive (resp., non-negative) when the Wasserstein distance between the two neighborhoods is strictly less than (resp., no larger than) the combinatorial distance between $x$ and $y$. A locally finite graph $G$ is said to have positive (resp., non-negative) Lin--Lu--Yau curvature if every edge in $G$ has positive (resp., non-negative) Lin--Lu--Yau curvature. Moreover, a graph has positive (resp. non-negative) Lin--Lu--Yau curvature precisely when it has positive (resp. non-negative) $\frac{1}{2}$-Ollivier Ricci curvature, see \eqref{eq:K_LLY} below. There is a large body of literature on various properties of graphs under Ollivier or Lin--Lu--Yau curvature bounds, see, e.g., \cite{BJL12,BM15,BCLMP18,Cush20,HM25,HLX24,JL14,LY10,LW20,MW19,Button3,Salez22,Smith14}.

The interaction between vertex- and edge-connectivity and Lin--Lu--Yau curvature of connected graphs has been explored by Chen, the first named author and You \cite{CLY25}. In particular, they establish the following relationship between the Lin--Lu--Yau curvature and edge-connectivity. Notice that all graphs in this paper are locally finite and simple. 
\begin{theorem}\cite{CLY25}\label{theorem:positive Lin--Lu--Yau curvature and edge-connectivity} Let $G$ be a connected graph with minimum vertex degree $\delta(G)$ and edge-connectivity $k'(G)$. If G has positive Lin--Lu--Yau curvature, then $k'(G)=\delta(G)$.
\end{theorem}

Observe that one can not weaken the assumption of \emph{positive} Lin--Lu--Yau curvature in Theorem \ref{theorem:positive Lin--Lu--Yau curvature and edge-connectivity} to be \emph{non-negative} Lin--Lu--Yau curvature. For example, the infinite path graph has non-negative Lin--Lu--Yau curvature, yet its edge-connectivity equals $1$ despite having a constant vertex degree of $2$. Chen, the first named author and You \cite{CLY25} posed the question of  whether any \emph{finite} connected graph with non-negative Lin--Lu--Yau curvature satisfies $k'(G)=\delta(G)$.

In this paper, we completely answer this problem. 

\begin{theorem}
\label{cor:finite and non-regular}
Let $G$ be a \emph{finite} connected graph with minimum vertex degree $\delta(G)$ and edge-connectivity $k'(G)$. If $G$ has non-negative Lin--Lu--Yau curvature, then $k'(G)=\delta(G)$.
\end{theorem}

We remark that graphs with positive or non-negative Lin--Lu--Yau curvature can be infinite. It is quite difficult to apply the global finiteness assumption to improve the original proof of Theorem \ref{theorem:positive Lin--Lu--Yau curvature and edge-connectivity} in \cite{CLY25}. To establish Theorem \ref{cor:finite and non-regular}, we find an approach different from that used in \cite{CLY25} and first show the following estimate.


\begin{theorem}\label{thm:non-regular and non-negative Lin--Lu--Yau curvature}
Let $G$ be a connected graph with minimum vertex degree $\delta(G)$ and edge-connectivity $k'(G)$. If $G$ has non-negative Lin--Lu--Yau curvature, then 
\begin{equation}\label{eq:k'(G) geq delta(G)-1}
    k'(G)\geq \delta(G)-1. 
\end{equation}
\end{theorem}
Notice that, the above estimate tells that any connected graph with non-negative Lin--Lu--Yau curvature satisfies $k'(G)\in \{\delta(G)-1,\delta(G)\}$. 
Furthermore, we characterize all graphs achieving the equality in Theorem \ref{thm:non-regular and non-negative Lin--Lu--Yau curvature}.  
\begin{theorem}\label{thm:non-regular rigidity}
Let $G$ be a connected graph with minimum vertex degree $\delta(G)$ and edge-connectivity $k'(G)$. Then $G$ has non-negative Lin--Lu--Yau curvature with $k'(G)= \delta(G)-1$ if and only if
\begin{equation}\label{eq:thm1.4}
G \in 
\begin{cases} 
\mathcal{G}_{k'(G)}, & \text{if } k'(G)=1,\ 2\text{ or }k'(G)\geq5, \\
\{G_3^*\}\cup \mathcal{G}_{3}, & \text{if } k'(G)=3, \\
\{G_4^*\}\cup \Tilde{\mathcal{G}}_{4}, & \text{if } k'(G)=4.
\end{cases}
\end{equation}
\end{theorem}

Theorem \ref{thm:non-regular rigidity} is our main contribution. The sets of graphs, $\mathcal{G}_n, n\in \mathbb{Z}_{>0}$ and $\Tilde{\mathcal{G}}_4$, as well as the two specific graphs $G_3^*$ and $G_4^*$, will be defined in Subsection \ref{subsection:rigidity_notations}. Notably, $\mathcal{G}_1=\{P_\infty\}$, where $P_\infty$ denotes the infinite path graph, and moreover the Cartesian product $P_{\infty}\times K_n$, with $K_n$ the complete graph on $n$ vertices, belongs to the set $\mathcal{G}_n$. In fact, each graph listed in \eqref{eq:thm1.4} is infinite. Therefore, Theorem \ref{cor:finite and non-regular} is a direct consequence of Theorem \ref{thm:non-regular rigidity} together with Theorem \ref{thm:non-regular and non-negative Lin--Lu--Yau curvature}. 

Moreover, each graph listed in \eqref{eq:thm1.4} has exactly two ends. Hua and M\"unch \cite{HM25} proved that any connected graph with non-negative Lin--Lu--Yau curvature has at most two ends. Hence, our results indicate that any connected graph $G$ with non-negative Lin--Lu--Yau curvature and \emph{at most one end} satisfies $k'(G)=\delta(G)$. On the other hand, there are connected graphs with non-negative Lin--Lu--Yau curvature and exactly two ends for which $k'(G)=\delta(G)$.

We will outline the approach used to prove Theorem \ref{thm:non-regular and non-negative Lin--Lu--Yau curvature} and Theorem \ref{thm:non-regular rigidity} in Subsection \ref{subsection:proof}. 

Before diving into more details, we would like to remark on other related works. Horn, Purcilly and Stevens \cite{HPS25} established a lower bound estimate for the vertex-connectivity in terms of the minimum vertex degree and the Bakry--\'Emery curvature lower bounds of the graph. Bakry--\'Emery curvature is another discrete notion of Ricci curvature.  Notice that
Bakry--\'Emery curvature and Lin--Lu--Yau curvature can behave very differently. There are graphs whose Bakry--\'Emery curvature and Lin--Lu--Yau curvature have opposite signs. Chen, Koolen and the first named author \cite{CKL25} showed that every connected graph with non-negative Bakry--\'Emery curvature satisfies $k'(G)\geq \delta(G)-1$. Consequently, any connected regular graph with non-negative Bakry--\'Emery curvature and an even or infinite number of vertices has a perfect matching, see \cite[Theorem 1.1]{CKL25}. 
Indeed, this perfect matching statement is a direct consequence of the edge-connectivity estimate \(k'(G)\geq \delta(G)-1\) combined with Tutte's theorem \cite{Tutte1,Tutte2}. Theorem \ref{thm:non-regular and non-negative Lin--Lu--Yau curvature} is a counterpart of their result in terms of Lin--Lu--Yau curvature. Using the same argument, we obtain the following corollary. 
\begin{corollary}
    Let $G$ be a connected regular graph with an even or infinite number of vertices. If $G$ has non-negative Lin--Lu--Yau curvature, then $G$ has a perfect matching.
\end{corollary}

\subsection{Graphs listed in Theorem \ref{thm:non-regular rigidity}}\label{subsection:rigidity_notations}
We first prepare necessary notations.
Let $k,\ell$ be two positive integers. Recall that $K_k$ denotes the complete graph on $k$ vertices. Let $K_{1,k}$ denote the star graph with $k$ leaves.  
A \emph{double star} graph $ST_{k,\ell}$ is a tree with a central edge $e = x \sim y$, where $x$ has $k$ leaves attached and $y$ has $\ell$ leaves attached. For a graph $G$, the notation $kG$ represents $k$ disjoint copies of $G$. 

We write $V(G)$ and $E(G)$ for the vertex set and edge set of $G$, respectively. For any subsets $S,T\subset V(G)$, we denote by $G[S]$ the induced subgraph on $S$ in $G$; Let $E(S,T):=\{u\sim v:u\in S\ \text{and}\ v\in T\}$ be the set of edges between $S$ and $T$; We denote by $G[S,T]$ the subgraph of $G$ with $V(G[S,T])=S\cup T$ and $E(G[S,T])=E(S,T)$. Notably, we have $G[S]=G[S,S]$. Notice that the subgraph $G[S,T]$ does not always coincide with the edge-induced subgraph of $E(S,T)$. For example, in Figure \ref{fig:An example for graph G[S,T]}, the edge-induced subgraph of $E(S,T)$ is the complete bipartite graph $K_{2,2}$, whereas the subgraph $G[S,T]$ consists of $K_{2,2}$ and two additional isolated vertices.
\begin{figure}[htbp]
    \centering
    \includegraphics[width=0.4\textwidth]{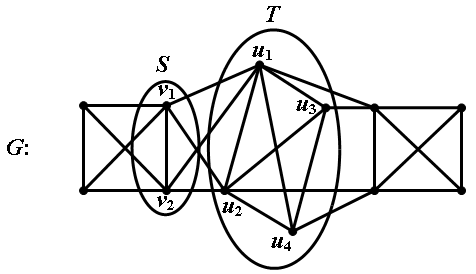}
    \caption{An example of the subgraph $G[S,T]$.}
    \label{fig:An example for graph G[S,T]}
\end{figure}

\begin{definition}[The graph set $\mathcal{G}_n$]\label{def:The graph class mathcal G_n}
    For any positive integer $n$, let $\mathcal{G}_n$ be the set of graphs $G$ satisfying the following properties: There exists a partition of the vertex set $V(G)=\sqcup_{i\in \mathbb{Z}}V_i$, such that
\begin{itemize}
    \item [(i)] $G[V_i]\in \{K_1,K_n\}$ for any $i\in \mathbb{Z}$;
    \item[(ii)] $E(V_i,V_j)=\emptyset$ for any $|i-j|>1$;
    \item[(iii)] $G[V_i,V_{i+1}]\in \{K_{1,n},nK_{1,1}\}$ for any $i\in \mathbb{Z}$.
\end{itemize}
\end{definition}
Note that $
\mathcal{G}_1=\{P_{\infty}\}$ and $P_\infty\times K_n\in \mathcal{G}_n$ for any positive integer $n$. Next, we define the two graphs $G_3^*$ and $G_4^*$.
\begin{definition}[The graphs $G_3^*$]\label{def:The graphs G_3^*}
    We define $G_3^*$ to be the $4$-regular graph satisfying the following properties: There exists a partition of the vertex set $V(G)=\sqcup_{i\in \mathbb{Z}}V_i$, such that 
\begin{itemize}
    \item [(i)] $G[V_i]=K_2$ for any $i\in \mathbb{Z}$;
    \item[(ii)] $E(V_i,V_j)=\emptyset$ for any $|i-j|>1$;
    \item[(iii)] $G[V_i,V_{i+1}]=ST_{1,1}$ for any $i\in\mathbb{Z}$.
\end{itemize}
\end{definition}
\begin{definition}[The graphs $G_4^*$]\label{def:The graphs G_4^*}
We define $G_4^*$ to be the $5$-regular graph satisfying the following properties: There exists a partition of the vertex set $V(G)=\sqcup_{i\in \mathbb{Z}}V_i$, such that 
\begin{itemize}
    \item [(i)] $G[V_i]=K_4$ for any $i\in \mathbb{Z}$;
    \item[(ii)] $E(V_i,V_j)=\emptyset$ for any $|i-j|>1$;
    \item[(iii)]  The edge-induced subgraph of $E(V_i,V_{i+1})$ is isomorphic to $2K_{1,2}$ with $3$ vertices in $V_i$ and the other $3$ vertices in $V_{i+1}$.
\end{itemize}
\end{definition}
The two graphs $G_3^*$ and $G_4^*$ are depicted in Figure \ref{fig:G3*} and Figure \ref{fig:G4*}.
\begin{figure}[H]
    \begin{minipage}{6cm}
    \centering
    \includegraphics[height=2cm,width=5cm]{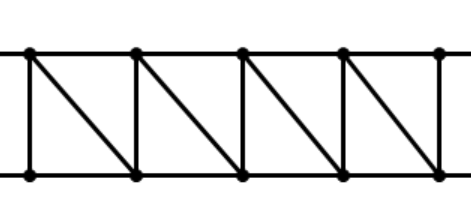}
    \caption{$G_3^*$}\label{fig:G3*}
    \end{minipage}
    \begin{minipage}{7cm}
    \centering
    \includegraphics[height=2cm,width=7cm]{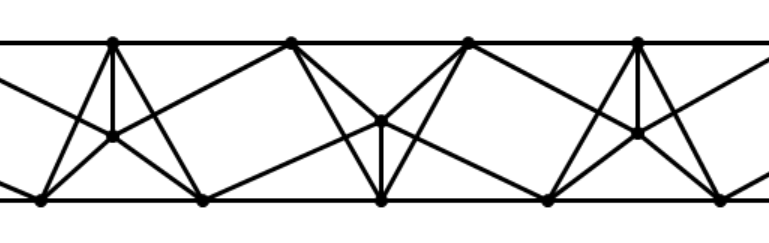}
    \caption{$G_4^*$}\label{fig:G4*}
    \end{minipage}
\end{figure}

The graph set $\Tilde{\mathcal{G}}_4$ is defined as follows.
\begin{definition}[The graph set $\Tilde{\mathcal{G}}_4$]\label{def:The graph set Tilde mathcal G_4}
    We define $\Tilde{\mathcal{G}}_4$ as the set of graphs $G$ satisfying the following properties: There exists a partition of the vertex set $V(G)=\sqcup_{i\in \mathbb{Z}}V_i$, such that 
\begin{itemize}
    \item [(i)] $G[V_i]\in \{K_1,K_2,K_4\}$ for any $i\in \mathbb{Z}$;
    \item[(ii)] $E(V_i,V_j)=\emptyset$ for any $|i-j|>1$;
    \item[(iii)] $G[V_i,V_{i+1}]\in \{K_{1,4},4K_{1,1},K_{2,2},2K_{1,2}\}$  for any $i\in\mathbb{Z}$.
    \item[(iv)] $(|V_{i-1}|,|V_i|,|V_{i+1}|)\neq (4,2,4)$ for any $i\in\mathbb{Z}$.
\end{itemize}
\end{definition}

\begin{remark}\label{remark: crucial condition that (|V_i-1|,|V_i|,|V_i+1|)neq (4,2,4)}
Condition (iv) in Definition \ref{def:The graph set Tilde mathcal G_4} excludes the existence of edges with negative Lin--Lu--Yau curvature in the graph. Let $G$ be a graph whose vertex set can be partitioned as $V(G)=\sqcup_{i\in\mathbb{Z}} V_i$ satisfying conditions (i)--(iii). Suppose that for some index $i$ we have $(|V_{i-1}|, |V_i|, |V_{i+1}|) = (4,2,4)$. Condition (iii) implies that both $G[V_{i-1}, V_i]$ and $G[V_i, V_{i+1}]$ are isomorphic to $2K_{1,2}$, and the resulting local structure is depicted in Figure \ref{fig:negative_curvature}. Notice that each edge in $E(V_{i-1}, V_i)$ or $E(V_i, V_{i+1})$ has Lin--Lu--Yau curvature equal to $-0.2$.
\end{remark}
\begin{figure}[htbp]  
    \centering  
    \includegraphics[height=3.3cm,width=6.5cm]{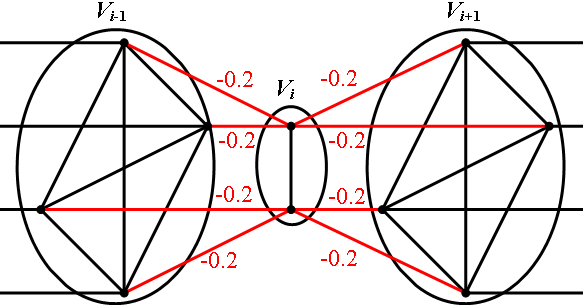}
    \caption{The local structure with $(|V_{i-1}|, |V_i|, |V_{i+1}|) = (4,2,4)$.}
    \label{fig:negative_curvature} 
\end{figure}

Observe that $\mathcal{G}_4$ is a subset of $\Tilde{\mathcal{G}}_4$. To illustrate the difference between these two sets, we depict in Figure~\ref{fig:An example of a graph in mathcal G_4.} a graph that belongs to $\mathcal{G}_4$ but is not the Cartesian product $P_\infty \times K_n$, whereas in Figure~\ref{fig:An example of a graph in Tilde mathcal G_4} we depict a graph in $\Tilde{\mathcal{G}}_4 \setminus \mathcal{G}_4$.

\begin{figure}[H]
    \begin{minipage}{7.3cm}
    \centering
    \includegraphics[height=2cm,width=7cm]{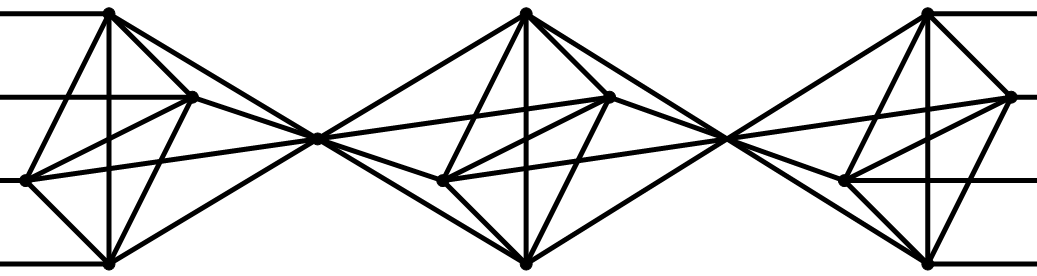}
    \caption{A graph in $\mathcal{G}_4$.}
    \label{fig:An example of a graph in mathcal G_4.}
    \end{minipage}
    \begin{minipage}{7.3cm}
    \centering
    \includegraphics[height=2cm,width=7cm]{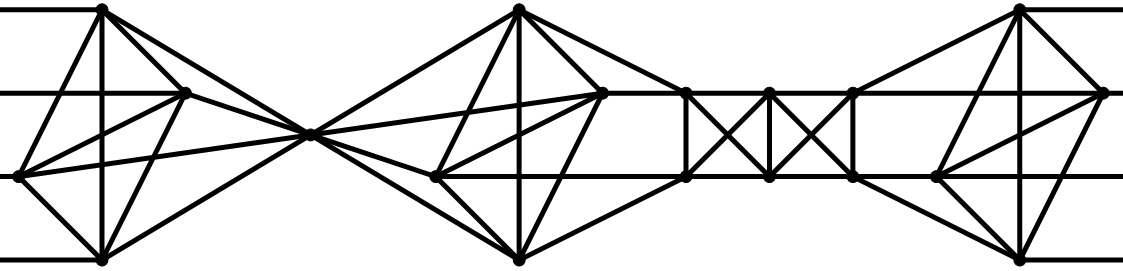}
    \caption{A graph in $\Tilde{\mathcal{G}}_4\setminus \mathcal{G}_4$.}
    \label{fig:An example of a graph in Tilde mathcal G_4}
    \end{minipage}

\end{figure}

\subsection{Strategy of proving the main results}\label{subsection:proof}
We begin by presenting a combinatorial inequality for graphs without isolated vertices and examining its rigidity properties. This inequality serves as a crucial tool for proving the edge-connectivity bound in Theorem~\ref{thm:non-regular and non-negative Lin--Lu--Yau curvature} and for addressing the rigidity problem discussed in Theorem~\ref{thm:non-regular rigidity}.
For an edge $e \in E(G)$, define
\[
S_1(e)_{(G)} := \{ \epsilon \in E(G) \;|\; \text{$e$ and $\epsilon$ share a common vertex} \}.
\]
The inequality and its rigidity characterization can be described as follows.

\begin{theorem}\label{thm:combinatorial}
Let $H$ be a finite graph with no isolated vertices, and suppose $H$ is neither a star graph nor $K_3$. Then
\begin{equation}\label{eq:combinatorial}
\min_{e \in E(H)} |S_1(e)_{(H)}| \leq |E(H)| - \frac{|V(H)|}{2}.
\end{equation}
Moreover, equality in \eqref{eq:combinatorial} holds if and only if $H$ is one of the following:
\[
K_{2,2}, \ K_4, \ K_4 - e, \ K_3 + e, \ nK_{1,1} \ \text{for } n \geq 2, \ 2K_{1,\frac{n}{2}} \ \text{for even } n > 2, \ ST_{\frac{n-1}{2},\frac{n-1}{2}} \ \text{for odd } n \geq 3.
\]
\end{theorem}

Here, $K_4 - e$ denotes the graph obtained by removing one edge from $K_4$, while $K_3 + e$ is formed by adding a new vertex and connecting it to one of the vertices of $K_3$.

Another crucial ingredient is a Lin--Lu--Yau curvature estimate for an edge in a cut. Let $G$ be a connected graph and let $E_0$ be a min-cut. Let $H$ be the edge-induced subgraph of $G$ induced by the edges in $E_0$. This subgraph $H$ is \emph{bipartite} and contains no isolated vertices. In Theorem~\ref{thm:estimate of k_LLY in a min-cut}, we will show that for any edge $e = x \sim y \in E_0 = E(H)$ with $d_x \geq d_y$, the Lin--Lu--Yau curvature $\kappa_{LLY}(e)$ satisfies
\begin{equation}\label{eq:curvature_combinatorial}
    d_x d_y \kappa_{LLY}(e) + d_x\left( 2|E(H)| - |V(H)| - 2|S_1(e)_{(H)}| \right) \leq 2d_x\left( |E(H)| +1 - d_y \right).
\end{equation}

By combining this inequality with Theorem~\ref{thm:combinatorial}, we deduce that if $G$ has non-negative Lin--Lu--Yau curvature and $H$ is not a star graph, then
\[
k'(G) = |E(H)| \geq \delta(G) - 1.
\]
Note that $H$ cannot be $K_3$ since it is bipartite. In the case where $H$ \emph{is} a star graph, a sharper curvature bound is required (see Theorem~\ref{thm:estimate of k_LLY in a star-cut}). All together lead to a proof of Theorem~\ref{thm:non-regular and non-negative Lin--Lu--Yau curvature}.

Now, suppose $G$ is connected with non-negative Lin--Lu--Yau curvature and $k'(G) = \delta(G) - 1$. Then equality must hold in both~\eqref{eq:combinatorial} and~\eqref{eq:curvature_combinatorial}. The rigidity statement in Theorem~\ref{thm:combinatorial} implies that the edge-induced subgraph $H$ of a min-cut in $G$ can only be one of the following:
\[
K_{2,2}, \ nK_{1,1} \ \text{for } n \geq 2, \ 2K_{1,\frac{n}{2}} \ \text{for even } n > 2, \ ST_{\frac{n-1}{2},\frac{n-1}{2}} \ \text{for odd } n \geq 3 \ \text{or} \ K_{1,n} \ \text{for } n \geq 1.
\]

The rigidity result corresponding to~\eqref{eq:curvature_combinatorial} (see Theorem~\ref{thm:estimate of k_LLY in a min-cut}) provides further insight into the local configuration around $H$. Moreover, the above set of possible graphs for $H$ can be further reduced to be (see Lemma \ref{lemma:admissible_graph_H})
\[
K_{2,2}, \ nK_{1,1} \ \text{for } n \geq 2, \ 2K_{1,2}, \ ST_{1,1}\ \text{and} \ K_{1,n} \ \text{for } n \geq 1.
\]

The remaining question is whether these local patterns can be extended to the entire graph. 
We will explore this issue in detail in Section~\ref{section:extension}.
The interplay between curvature bounds for edges in a cut and the combinatorial inequality~\eqref{eq:combinatorial} is particularly noteworthy.

The rest of this paper is organized as follows: In Section \ref{section:preliminary}, we introduce the notations and review basics about Lin--Lu--Yau curvature of graphs. Section \ref{section:combinatorial} contains the proof of Theorem \ref{thm:combinatorial}. In Section \ref{section:curvature}, we establish curvature estimates of edges in a cut, see Theorem \ref{thm:estimate of k_LLY in a min-cut} and Theorem \ref{thm:estimate of k_LLY in a star-cut}. In particular, we prove \eqref{eq:curvature_combinatorial}. Finally, in Section \ref{section:extension}, we complete the proofs of Theorem \ref{thm:non-regular and non-negative Lin--Lu--Yau curvature} and Theorem \ref{thm:non-regular rigidity}. In Appendix \ref{subsection:non-negative LLY curvature}, we check that each graph listed in Theorem \ref{thm:non-regular rigidity} has non-negative Lin--Lu--Yau curvature.

\section{Preliminaries}\label{section:preliminary}

Let $G$ be a graph with vertex set $V(G)$ and edge set $E(G)$. For any $u,v\in V(G)$, we denote by $d(u,v)$ the combinatorial distance between $u$ and $v$, and define 
\[S_1(u)=\{w\in V : d(u,w)=1\}, \text{and} \ B_1(u)=\{w\in V : d(u,w)\leq 1\}.\]
The cardinality of $S_1(u)$ is the vertex degree $d_u$ of $u$.  For any $u\in V(G)$ and $X\subset V(G)$, we define $d(u,X):=\min_{v\in X}d(u,v)$. 
Throughout this paper, all graphs are simple and locally finite, that is, there are no self-loops or multi-edges and the degree of each vertex is finite. 

Next, we recall basics about Ollivier Ricci curvature and Lin--Lu--Yau curvature of graphs. The definition is based on the  concept of the Wasserstein distance between probability measures.

\begin{definition}[\textbf{Wasserstein Distance}]\label{definition:Wasserstein}
    Let $G$ be a graph and $\mu_1,\mu_2$ be two probability measures on $V(G)$. The Wasserstein distance $W(\mu_1, \mu_2)$ is defined as
    \[
    W(\mu_1, \mu_2) = \inf_{\pi} \sum_{u \in V(G)} \sum_{v \in V(G)} d(u, v)  \pi(u, v),
    \]
    where the infimum is taken over all transport plans $\pi: V(G) \times V(G) \to [0, 1]$ satisfying the marginal constraints:
    \[
    \sum_{v \in V(G)} \pi(u, v) = \mu_1(u) \quad \text{and} \quad \sum_{u \in V(G)} \pi(u, v) = \mu_2(v), \ \text{for any } u, v \in V(G).
    \]
\end{definition}
The Kantorovich duality  tells that (see, e.g., ~\cite[Theorem~1.14]{Villani03})
\begin{equation}\label{eq:Kantorovich}
W(\mu_1, \mu_2) = \sup_{\substack{f: V(G) \to \mathbb{R} \\ \text{$1$-Lipschitz}}}  \sum_{u \in V(G)} f(u) \left(\mu_1(u) - \mu_2(u)\right),
\end{equation}
where the supremum is taken over all $1$-Lipschitz functions \(f\). Recall that a function $f:V(G)\to \mathbb{R}$ is called $1$-Lipschitz if $|f(x) - f(y)| \leq d(x, y)$ holds for any $u,v \in V(G).$

For a vertex $x$ and an idleness parameter $\rho \in [0,1]$, we define the probability measure $\mu_x^{\rho}$ as
\begin{equation}\label{eq:mu_rho}
\mu_x^{\rho}(v) = 
\begin{cases} 
\rho, & \text{if } v = x, \\
\frac{1-\rho}{d_x}, & \text{if } v \sim x, \\
0, & \text{otherwise.}
\end{cases}
\end{equation}

\begin{definition}[Ollivier Ricci curvature \cite{Ollivier09,Button3} and Lin--Lu--Yau curvature \cite{LLY11}]\label{def:k_LLY}
    Let $G$ be a graph. For $x, y \in V(G)$, the $\rho$-Ollivier Ricci curvature is
    \[
    \kappa_{\rho}(x,y) = 1 - \frac{W(\mu_x^{\rho}, \mu_y^{\rho})}{d(x,y)}.
    \]
    The Lin--Lu--Yau curvature $\kappa_{LLY}(x,y)$ is then defined by the limit
    \[
    \kappa_{LLY}(x,y) = \lim_{\rho \to 1} \frac{\kappa_{\rho}(x,y)}{1 - \rho}.
    \]
\end{definition}

Notice that $\kappa_{LLY}(x,y) = -\left. \frac{d \kappa_{\rho}(x,y)}{d \rho} \right|_{\rho=1}$ since $\kappa_1(x,y) = 0$.
For an edge $e = x \sim y$, it has been shown in \cite{BCLMP18} that the function $\rho\mapsto \kappa_{\rho}(x,y)$ is linear on the interval $\left[\frac{1}{\max\{d_x,d_y\}+1}, 1\right]$. This yields a limit-free expression for $\kappa_{LLY}(e)$:
 \begin{equation}\label{eq:K_LLY}
      \kappa_{LLY}(e) = \frac{\kappa_{\rho}(x,y)}{1 - \rho} \\ 
      =\frac{1}{1-\rho} \left(1 - W\left(\mu_x^{\rho}, \mu_y^{\rho}\right)\right), \ \ \text{for any}\ \ \rho\in \left[\frac{1}{\max\{d_x,d_y\}+1}, 1\right).
\end{equation}

In case that $d_x=d_y=:\beta$, we have $\mu_x(u),\mu_y(u)\in\{0,\frac{1}{\beta+1}\}$ for any $u\in V(G)$. Then it is a Monge problem to calculate $W(\mu_x,\mu_y)$. That is, an optimal transport plan $\pi$ between $\mu_{x}$ and $\mu_{y}$ can be induced by a bijective map $\phi: B_1(x)\to B_1(y)$ via 
\begin{equation}\label{eq:pi_phi}
    \pi(u,\phi(u))=\frac{1}{\beta+1}
\end{equation} for any $u\in B_1(x)$, see \cite[Proposition 2.7]{Cush20} and \cite[Theorem 2.6]{MW19}. By \cite[Lemma 4.1]{BCLMP18}, we can choose a map $\phi$ for an optimal transport plan such that $\phi(u)=u$ for any $u\in B_1(x)\cap B_1(y)$. That is, we have for an edge $e=x\sim y$ with $d_x=d_y=:\beta$ that
\begin{equation}\label{eq:W_transport_map}
    (\beta+1)\cdot W(\mu_x,\mu_y)=\min_{\psi: B_1(x)\setminus B_1(y)\to B_1(y)\setminus B_1(x)}\sum_{u\in B_1(x)\setminus B_1(y)}d(u,\psi(u)),
\end{equation}
where the minimum is taken over all bijective maps between $B_1(x)\setminus B_1(y)$ and $B_1(y)\setminus B_1(x)$. 

Recall that by triangle inequality for Wasserstein distance, $\kappa_{LLY}(x,y)\geq 0$ for all $x,y\in V(G)$ if and only if $\kappa_{LLY}(e)\geq0$ for all $e\in E(G)$, see \cite[Lemma 2.3]{LLY11}.

\section{A combinatorial inequality and its rigidity}\label{section:combinatorial}

In this section, we prove Theorem  \ref{thm:combinatorial}.

\begin{proof} We first prove the inequality \eqref{eq:combinatorial} by contradiction. Suppose that \eqref{eq:combinatorial} does not hold. That is, we have for each edge $e=u\sim v\in E(H)$ that
\begin{equation}\label{eq:half}
    (d_u-1)+(d_v-1)=|S_1(e)_{(H)}|>|E(H)|-\frac{|V(H)|}{2}=\frac{1}{2}\sum_{w\in V(H)}(d_w-1).
\end{equation}
Since $H$ has no isolated vertices and is not a star graph, $H$ has at least two edges. If any two edges of $H$ have a common vertex, i.e., the line graph of $H$ is a complete graph, then $H$ is either a star graph or $K_3$. Therefore, by our assumption, there exist two disjoint edges $u_i\sim v_i$, $i=1,2$ in $E(H)$. Applying \eqref{eq:half} leads to
\[\sum_{w\in V(H)}(d_w-1)\geq \sum_{i=1}^2\left[(d_{u_i}-1)+(d_{v_i}-1)\right]>\sum_{w\in V(H)}(d_w-1),\]
which is a contradiction. This proves \eqref{eq:combinatorial}.

Assume that the equality holds in \eqref{eq:combinatorial}, that is, $
    \min_{e\in E(H)}|S_1(e)_{(H)}|=|E(H)|-\frac{|V(H)|}{2}$. 
Then we derive by the above argument that
\begin{equation}\label{eq:identity}
    \sum_{w\in V(H)}(d_w-1)= \sum_{i=1}^2\left[(d_{u_i}-1)+(d_{v_i}-1)\right].
\end{equation}

Since $H$ has no isolated vertices and is not a star graph or $K_3$, we have $|V(H)|\geq 4$.
If $|V(H)|=4$, there are only $6$ possible graphs with no isolated vertices which are not star graphs. It is direct to check that the equality in \eqref{eq:combinatorial} holds for all of them: $2K_{1,1},K_4, K_4-e, K_{2,2}, K_3+e$ and $ST_{1,1}$. Note that $ST_{1,1}$ is a path graph with $4$ vertices. 

Next, we assume $|V(H)|>4$. Due to \eqref{eq:identity}, every vertex in $V(H)\setminus\{u_1,u_2,v_1,v_2\}$ is of degree $1$.
If there exists an edge $e=x\sim y$ such that $x,y\in V(H)\setminus\{u_1,u_2,v_1,v_2\}$, then we derive from the equality in \eqref{eq:combinatorial} that
\begin{equation}\label{eq:K2}
    0=|S_1(e)|\geq |E(H)|-\frac{|V(H)|}{2}\geq 0. 
\end{equation}
This forces that $2|E(H)|=\sum_{w\in V(H)}d_w=|V(H)|$. Therefore, every vertex in $H$ is of degree $1$, that is,  $H=nK_{1,1}$ with $n\geq 3$.

If, otherwise, there does not exist an edge $e=x\sim y$ with $x,y\in V(H)\setminus\{u_1,u_2,v_1,v_2\}$, then every vertex in $V(H)\setminus\{u_1,u_2,v_1,v_2\}$ is connected to one of $\{u_1,u_2,v_1,v_2\}$ via an edge. In particular, $H$ has at most $2$ connected components. Moreover, since at least one of  $\{u_1,u_2,v_1,v_2\}$ has degree at least $2$, we have $2|E(H)|>|V(H)|$. Denote by \[L:=\{u\in V(H): d_u=1\}\ \text{and}\ S:=\{w\in V(H): \text{there exists } u\in L \ \ \text{such that}\ \ w\sim u\}.\] Note that $S\subset \{u_1,u_2,v_1,v_2\}$ is not empty. Then we deduce from the equality in \eqref{eq:combinatorial} that
\[d_w-1\geq |E(H)|-\frac{|V(H)|}{2},\ \text{for any }w\in S.\] Therefore, we derive
\[2|E(H)|=\sum_{w\in V(H)}d_w\geq |L|+|S|\left(|E(H)|-\frac{|V(H)|}{2}+1\right)+2|V\setminus(L\cup S)|.\]
Rearranging leads to 
\[|S|\leq 2-\frac{2|V(H)\setminus(L\cup S)|}{2|E(H)|-|V(H)|}.\]
Notice that $|S|>1$. Indeed, if $|S|=1$, there exists $i\in \{1,2\}$ such that $d_{u_i}=d_{v_i}=1$. By an argument similar as in \eqref{eq:K2}, we have $|V(H)|=2|E(H)|$, which is a contradiction. Therefore, we have $|S|=2$ and $V(H)\setminus{(L\cup S)}=\emptyset$. This tells that $H$ is a forest. 
In the following, we divide our discussion into two cases.

\textbf{Case 1:} $H$ has $2$ connected components $H_1$ and $H_2$. Then we have $|V(H)|-|E(H)|=2$. Let $x_i$ be a leaf of $H_i$ and $y_i\in V(H_i)$ be a neighbor of $x_i$. Denote the edge by $e_i=x_i\sim y_i$. Then
\[
|V(H_i)|\geq d_{x_i}+d_{y_i}=1+d_{y_i}=|S(e_i)|+2\geq |E(H)|-\frac{|V(H)|}{2}+2=\frac{|E(H)|}{2}+1,\ \ \text{for } i=1,2.
\]
Summing up leads to
\[|E(H)|+2=|V(H)|=|V(H_1)|+|V(H_2)|\geq  |E(H)|+2.\]
This forces that $V(H_i)=1+d_{y_i}=1+\frac{|E(H)|}{2}$.
Hence, $|E(H)|$ is even and $H_i$ is a star graph with $\frac{|E(H)|}{2}+1$ vertices. That is, we have $H=2K_{1,\frac{n}{2}}$ with $n=|E(H)|>2$ even.

\textbf{Case 2:} $H$ is connected, that is, $H$ is a tree. In this case, we have $|V(H)|-|E(H)|=1$. For convenience, we denote by $S:=\{y_1,y_2\}\subset\{u_1,u_2,v_1,v_2\}$. Then, we must have $y_1\sim y_2$. Since both are connected to a vertex of degree $1$, we estimate
\[1+d_{y_i}\geq |E(H)|-\frac{|V(H)|}{2}+2=\frac{|V(H)|}{2}+1.
\]
Summing up, we deduce
\[|V(H)|\geq d_{y_1}+d_{y_2}\geq |V(H)|.\]
This forces that $d_{y_i}=\frac{|V(H)|}{2}$, $i=1,2$.
This tells that $H=ST_{\frac{n-1}{2},\frac{n-1}{2}}$ with $n=|V(H)|-1=|E(H)|$. Since $d_{y_i}$ is an integer, $n=|E(H)|=|V(H)|-1>3$ is odd.

In conclusion, if the equality in \eqref{eq:combinatorial} holds, then $H$ is one of the following graphs: $K_{2,2}$, $K_4$, $K_4 - e$, $K_3+e$; $nK_{1,1}$ for $n\geq 2$, $2K_{1,\frac{n}{2}}$ for $n> 2$ even, $ST_{\frac{n-1}{2},\frac{n-1}{2}}$ for $n\geq 3$ odd. It is straightforward to check the equality in \eqref{eq:combinatorial} does hold for all those graphs. This completes the proof.
\end{proof}
\section{Lin--Lu--Yau curvature of edges in a cut}\label{section:curvature}

In this section, we prove the following theorem relating \eqref{eq:combinatorial} to the Lin--Lu--Yau curvature.
\begin{theorem}\label{thm:estimate of k_LLY in a min-cut}
Let $G=(V,E)$ be a connected graph. Let $E_0\subset E$ be a cut of $G$ and $H$ be the edge-induced subgraph on $E_0$ of $G$. For any edge $e=x\sim y\in E_0=E(H)$, we have
\begin{equation}\label{eq:the estimate of k_LLY with d_x>=d_y}
   d_xd_y\kappa_{LLY}(e) \leq \max\{d_x,d_y\}\left(|V(H)|+2|S_1(e)_{(H)}|\right)+2\min\{d_x,d_y\}-2d_xd_y.
\end{equation}
If equality is achieved, then every vertex \( u \in A \setminus \{x\} \) (resp., $v\in B\setminus\{y\}$) is adjacent to \( x \) (resp., $y$). 
Here, \( A \) (resp., \( B \)) denotes the set of vertices in \( V(H) \) that lie in the same connected component as \( x \) (resp., \( y \)) once the cut \( E_0 \) is removed.
\end{theorem}
\begin{proof} Recall that $S_1(e)_{(H)}$ is the set of edges in $E_0=E(H)$ incident to the edge $e=x\sim y$.
We denote by \[M:=\left\{\text{end vertices of edges in } S_1(e)_{(H)}\right\}\setminus\{x,y\}.\] Note that $|M|=|S_1(e)_{(H)}|$.

Without loss of generality, we assume that $d_x\geq d_y$. Denote by $a:=\frac{1}{d_x(d_y+1)}$. Consider the following two probability measures as defined in \eqref{eq:mu_rho}:
\[
\mu_x^{\frac{1}{d_y+1}}(u) = 
\begin{cases} 
d_x a, & \text{if } u = x, \\
d_y a, & \text{if } u \sim x, \\
0, & \text{otherwise,}
\end{cases}\ \ \text{and}\ \ \mu_y^{\frac{1}{d_y+1}}(v) = 
\begin{cases} 
d_x a, & \text{if } v = y, \\
d_x a, & \text{if } v \sim y, \\
0, & \text{otherwise,}
\end{cases}
\]
Let $\pi$ be a transport plan such that 
\begin{equation}\label{eq:pi_assumption}
    W\left(\mu_x^{\frac{1}{d_y+1}}, \mu_y^{\frac{1}{d_y+1}}\right)=\sum_{u\in B_1(x)}\sum_{v\in B_1(y)}\pi(u,v)d(u,v).
\end{equation}
By \cite[Lemma 4.1]{BCLMP18}, we can further assume that \[\pi(x,x)=d_xa,\ \pi(y,y)=d_ya\ \text{and}\ \pi(u,u)=d_ya\   \text{for any}\  u\in S_1(x)\cap S_1(y).\] 
This implies that $\pi(x,v)=\pi(v,x)=0$ for any $v\neq x$, and $\pi(y,v)=0$ for any $v\neq y$.
Let us define the following sets:
\begin{align*}
Z_1:=&\left\{(u,v)\in B_1(x)\setminus\{x,y\}\times B_1(y)\setminus\{x,y\}:  u\in M \right\},
\\
Z_2:=&\left\{(u,v)\in B_1(x)\setminus\{x,y\}\times B_1(y)\setminus\{x,y\}:  u\not\in M\ \text{and}\ v\in M \right\},
\\
  \text{and}\  Z_3:=&\left\{(u,v)\in B_1(x)\setminus\{x,y\}\times B_1(y)\setminus\{x,y\}: u\not\in M\ \text{and}\ v\not\in M\right\}.
\end{align*}
By definition, we have
\begin{align}
    1=&d_x(d_y+1)a=\pi(x,x)+\pi(y,y)+\sum_{u\in B_1(x)\setminus\{x,y\}}\pi(u,y)+\sum_{i=1}^3\sum_{(u,v)\in Z_i}\pi(u,v)\notag
    \\
    = &2d_x a+\sum_{i=1}^3\sum_{(u,v)\in Z_i}\pi(u,v).\label{eq:pi_estimate}
\end{align}
Observing from \eqref{eq:pi_assumption} that $\sum_{u\not\in M}\pi(u,v)\leq (d_x-d_y)a$ for any $v\in B_1(x)\cap B_1(y)\cap M$. This leads to the following estimate:
\begin{align}\label{eq:Z_1_2_estimate}
    &\sum_{i=1}^2\sum_{(u,v)\in Z_i}\pi(u,v)\notag\\
    \leq &d_ya|B_1(x)\cap M|+d_xa|\left(B_1(y)\cap M\right)\setminus B_1(x)|+(d_x-d_y)a|B_1(y)\cap B_1(x)\cap M|
    \leq  d_xa|M|.
\end{align}
Inserting \eqref{eq:Z_1_2_estimate} into \eqref{eq:pi_estimate} and rearranging yield that
\begin{equation}\label{eq:Z_2_lowerbound}
    \sum_{(u,v)\in Z_3}\pi(u,v)\geq d_x(d_y-1-|M|)a.
\end{equation}
We further consider the following two subsets of $Z_3$,
\[Z_{3,1}=\{(u,v)\in Z_3: d(u,v)=1\} \ \text{and}\ Z_{3,2}=\{u\in Z_3: d(u,v)=2\}.\]
Recall that $E_0=E(H)$ is a cut of $G$. Let $V(G)=X\sqcup Y$ be the partition of $V(G)$ such that $x\in X$, $y\in Y$ and $E(X,Y)=E_0=E(H)$. This leads to the following observations: For any $(u,v)\in Z_{3,1}$, we have $u\in X$, $v\in Y$ and hence the edge $u\sim v$ must lie in $E_0=E(H)$. That is, both $u$ and $v$ lie in $V(H)$; For any $u\in Z_{3,2}$, we have $u\in X$, $v\in Y$ and hence the shortest $2$-path connecting $u$ and $v$ in $G$ contains exactly one edge in $E_0=E(H)$. That is, at least one of $u$ and $v$ lies in $V(H)$. Then, we derive that
\begin{align}\label{eq:Z_31_32}
    &2\sum_{(u,v)\in Z_{3,1}}\pi(u,v)+\sum_{(u,v)\in Z_{3,2}}\pi(u,v)\notag\\
    \leq &d_ya|\left(B_1(x)\cap V(H)\right)\setminus\left(\{x,y\}\cup M\right)|+d_xa|
    \left(B_1(y)\cap V(H)\right)\setminus\left(\{x,y\}\cup M\right)|\notag
    \\\leq &d_xa \left(|V(H)|-2-|M|\right).
\end{align}
Now we estimate the optimal transportation distance
\begin{align}
    &W\left(\mu_x^{\frac{1}{d_y+1}}, \mu_y^{\frac{1}{d_y+1}}\right)=\sum_{u\in B_1(x)}\sum_{v\in B_1(y)}\pi(u,v)d(u,v)\notag\\
    \geq&\sum_{u\in B_1(x)\setminus \{x,y\}}\pi(u,y)d(u,y)+\sum_{(u,v)\in Z_3}\pi(u,v)d(u,v) \label{eq:drop_Z_1_Z_2}\\
    =&2(d_x-d_y)a+\sum_{i=1}^2\sum_{(u,v)\in Z_{3,i}}\pi(u,v)i+3\left(\sum_{(u,v)\in Z_3}\pi(u,v)-\sum_{i=1}^2\sum_{(u,v)\in Z_{3,i}}\pi(u,v)\right)\notag\\
    = & 2(d_x-d_y)a+3\sum_{(u,v)\in Z_3}\pi(u,v)-\sum_{i=1}^2\sum_{(u,v)\in Z_{3,i}}\pi(u,v)(3-i).\notag
\end{align}
Applying \eqref{eq:Z_2_lowerbound} and \eqref{eq:Z_31_32} leads to 
\begin{align}\label{eq:W_lower_bound}
    W\left(\mu_x^{\frac{1}{d_y+1}}, \mu_y^{\frac{1}{d_y+1}}\right)\geq 2(d_x-d_y)a+3d_x(d_y-1-|M|)a-d_xa(|V(H)|-2-|M|).
\end{align}
By \eqref{eq:K_LLY}, we obtain
\begin{align}
    d_xd_y\kappa_{LLY}(e)=d_x(d_y+1)\left(1-W\left(\mu_x^{\frac{1}{d_y+1}}, \mu_y^{\frac{1}{d_y+1}}\right)\right)
    \leq d_x\left(|V(H)|+2|M|\right)+2d_y-2d_xd_y.
\end{align}
This proves \eqref{eq:the estimate of k_LLY with d_x>=d_y}. Moreover, if the equality in \eqref{eq:the estimate of k_LLY with d_x>=d_y} holds, then the equalities in \eqref{eq:Z_1_2_estimate}, \eqref{eq:Z_31_32} and \eqref{eq:drop_Z_1_Z_2} must hold. Since the equality \eqref{eq:drop_Z_1_Z_2} holds, we have $\pi(u,v)d(u,v)=0$ for any $(u,v)\in Z_1\sqcup Z_2$. This tells in particular that $\pi(u,v)=0$ for any $(u,v)\in Z_2$, and $\pi(u,v)\neq 0$ only if $d(u,v)=0$. Then, we applying the equality in \eqref{eq:Z_1_2_estimate} to deduce
\begin{align*}
    d_xa|M|=\sum_{(u,v)\in Z_1\sqcup Z_2}\pi(u,v)=\sum_{(u,v)\in Z_1}\pi(u,v)=\sum_{u\in M}\pi(u,u)\leq d_ya|M|.
\end{align*}
Therefore, we have $d_x=d_y=\pi(u,u)$ for any $u\in M$ if $M\neq \emptyset$. This implies that $M\subset B_1(x)\cap B_1(y)$. 

We further apply the equality in \eqref{eq:Z_31_32} to conclude that $u\sim x$ for any $u\in A\setminus \{x\}$ and $v\sim y$ for any $v\in B\setminus \{y\}$, where $A:=X\cap V(H)$ and $B:=Y\cap V(H)$. This completes the proof of this theorem.
\end{proof}

\begin{theorem}\label{thm:estimate of k_LLY in a star-cut}
Let $G=(V,E)$ be a connected graph. Let $E_0\subset E$ be a cut of $G$ and $H$ be the edge-induced subgraph of $E_0$ in $G$. Suppose that $H$ is a star graph. Then we have for any edge $e=x\sim y\in E_0=E(H)$ with $y$ being a leaf in $H$ that,
    \begin{equation}\label{eq:H_star}
    d_x d_y\kappa_{LLY}(e)\leq d_x\left(\alpha_{x,y}+2-d_y\right)+d_y\left(2|S_1(e)_{(H)}|+2-d_x\right),
  \end{equation}
  where $\alpha_{x,y}:=|S_1(x)\cap S_1(y)|$.
\end{theorem}
\begin{remark}
Theorem \ref{thm:estimate of k_LLY in a star-cut} is an improvement of Theorem \ref{thm:estimate of k_LLY in a min-cut} for the case of $H$ being a star graph. Indeed, for any edge $e=x\sim y\in E_0=E(H)$ with $y$ being a leaf, we have
\[d_x(\alpha_{x,y}+2)+2d_y\leq \max\{d_x,d_y\}|V(H)|+2\min\{d_x,d_y\}.\]
In the above, we used $\alpha_{x,y}+2\leq |V(H)|$.  We point out that, in case that $H$ is not a star graph, Theorem \ref{thm:estimate of k_LLY in a min-cut} can not be improved as $d_xd_y\kappa_{LLY}(e) \leq \max\{d_x,d_y\}\left(\alpha_{x,y}+2+2|S_1(e)_{(H)}|\right)+2\min\{d_x,d_y\}-2d_xd_y.$ A counterexample is given by an edge not included in a $3$-cycle of $P_{\infty}\times K_3$.
\end{remark}
\begin{proof}
Let $V(G)=X\sqcup Y$ be the partition of $V(G)$ such that $E(X,Y)=E_0=E(H)$ and $x\in X$. Set $Z:=X\setminus\{x\}$. For convenience, we denote $a:=\frac{1}{d_x(d_y+1)}$. By Kantorovich duality \eqref{eq:Kantorovich}, we estimate
\begin{equation*}
  \begin{aligned}
    &W\left(\mu_y^{\frac{1}{d_y+1}},\mu_x^{\frac{1}{d_y+1}}\right)
    \geq \sum_{v \in V} d(v,Z) \left(\mu_y^{\frac{1}{d_y+1}}(v) -\mu_x^{\frac{1}{d_y+1}}(v)\right)\\
   =&2d_xa(\alpha_{x,y}+1)-2d_ya\left(|S_1(e)_{(H)}|+1\right)+3d_xa(d_y-1-\alpha_{x,y})
  \end{aligned}
\end{equation*}
Therefore, we deduce by (\ref{eq:K_LLY}) that
\begin{align*}
    d_xd_y\kappa_{LLY}(e)&=d_x(d_y+1) \left(1 - W\left(\mu_y^{\frac{1}{d_y+1}},\mu_x^{\frac{1}{d_y+1}}\right)\right)\\ 
    &\leq d_x\left(\alpha_{x,y}+2-d_y\right)+d_y\left(2|S_1(e)_{(H)}|+2-d_x\right).
\end{align*}
This completes the proof.
\end{proof}

\section{Proofs of the main results}\label{section:extension}

We prove Theorem \ref{thm:non-regular and non-negative Lin--Lu--Yau curvature} and Theorem \ref{thm:non-regular rigidity} in this section.

\begin{proof}[Proof of Theorem \ref{thm:non-regular and non-negative Lin--Lu--Yau curvature}]
Let $E_0\subset E$ be a min-cut of $G$ and $H$ be the edge-induced subgraph on $E_0$ of $G$. Then $H$ is a bipartite graph without isolated vertices. We divide our arguments into two cases.

\textbf{Case 1:} $H$ is a star graph. Let $e=x\sim y\in E_0=E(H)$ be an edge of $H$. Observe that $|S_1(e)_{(H)}|=|E(H)|-1=k'(G)-1$.
We assume without loss of generality that $y$ is a leaf vertex in $H$. 
By Theorem \ref{thm:estimate of k_LLY in a star-cut}, we derive that
\begin{equation*}
     d_xd_y\kappa_{LLY}(e)\leq d_x\left(\alpha_{x,y}+2-d_y\right)+d_y\left(2k'(G)-d_x\right).
\end{equation*}
Since $E_0=E(H)$ is a min-cut, we have $d_x\geq 2k'(G)$.
Therefore, we deduce that
\begin{equation}\label{eq:1.3_estimate_3}
   0\leq d_xd_y\kappa_{LLY}(e)+d_y(d_x-2k'(G))\leq d_x(\alpha_{x,y}+2-d_y).
\end{equation}
Rearranging yields that
\begin{equation}\label{eq:1.3_estimate_4}
    \delta(G)\leq d_y\leq \alpha_{x,y}+2\leq k'(G)+1.
\end{equation}

\textbf{Case 2:} $H$ is not a star graph. Let $e=x\sim y\in E_0=E(H)$ be an edge of $H$ attaining the minimum of $|S_1(e)_{(H)}|$. We assume without loss of generality that $d_x\geq d_y$. Applying Theorem \ref{thm:estimate of k_LLY in a min-cut}, we estimate
\begin{equation}\label{eq:1.3_estimate_1}
       0\leq d_xd_y\kappa_{LLY}(e)\leq d_x\left(|V(H)|+2|S_1(e)_{(H)}|\right)+2d_y-2d_xd_y
    \leq 2d_x|E(H)|+2d_x-2d_xd_y 
\end{equation}
In the second inequality above, we have applied Theorem \ref{thm:combinatorial} and the assumption $d_y\leq d_x$. Since $G$ has non-negative Lin--Lu--Yau curvature, we derive from \eqref{eq:1.3_estimate_1} that
\begin{equation}\label{eq:1.3_estimate_2}
    k'(G)=|E(H)|\geq d_y-1\geq \delta(G)-1.
\end{equation}
This completes the proof of Theorem \ref{thm:non-regular and non-negative Lin--Lu--Yau curvature}.
\end{proof}

Next, we prove Theorem \ref{thm:non-regular rigidity}. Let us first characterize possible local structures around the graph $H$ induced by a min-cut $E_0$.

\begin{lemma}\label{lemma:structure of G[]}
Let $G=(V,E)$ be a connected graph with non-negative Lin--Lu--Yau curvature and $k'(G)= \delta(G)-1$. Let $E_0\subset E$ be a min-cut of $G$ and $H$ be the edge-induced subgraph of $E_0$ in $G$.  Denote by $V(G)=X\sqcup Y$ the partition of $V(G)$ such that $E(X,Y)=E_0=E(H)$. Set $A:= X\cap V(H)$ and $B:=Y\cap V(H)$. 
\begin{itemize}
  \item [(1)] If $H$ is a star graph $K_{1,n}$ for $n\geq 1$, assuming without loss of generality that $A=\{x\}$ contains only one vertex $x$, then we have
    \begin{itemize}
      \item [(a)] $d_{x}=2k'(G)$;
      \item [(b)] $d_y=k'(G)+1$, for any $y\in B$;
      \item [(c)] $G[B]=K_{|B|}$.
    \end{itemize} 
  \item [(2)] If $H$ is not a star graph, then we have
    \begin{itemize}
    \item [(a)] $H\in \{K_{2,2}, \ nK_{1,1} \ \text{for } n \geq 2, \ 2K_{1,\frac{n}{2}} \ \text{for even } n > 2, \ ST_{\frac{n-1}{2},\frac{n-1}{2}} \ \text{for odd } n \geq 3\}$;
      \item [(b)]$d_x=d_y=k'(G)+1$ for any $x\in A$ and $y\in B$;
      \item [(c)] $G[A]=K_{|A|}$ and $G[B]=K_{|B|}$.
    \end{itemize} 
\end{itemize}  
\end{lemma}
\begin{proof} We first assume that $H$ is a star graph. For any $y\in B$, we apply the argument of \textbf{Case 1} in the proof of Theorem \ref{thm:non-regular and non-negative Lin--Lu--Yau curvature} to the edge $e=x\sim y\in E_0=E(H)$. Then the property $k'(G)=\delta(G)-1$ forces the equalities in \eqref{eq:1.3_estimate_4}, that is, $d_y=k'(G)+1=\alpha_{x,y}+2$. This proves $(1)(b)$ and $\alpha_{x,y}=k'(G)-1=|E(H)|-1$. The latter shows that $y$ is adjacent to all other vertices in $B$. This proves (1)(c).  Furthermore, the fact $k'(G)=\delta(G)-1$ forces the equalities in \eqref{eq:1.3_estimate_3}, that is $\kappa(e)=0$ and $d_x=2k'(G)$. The latter proves (1)(a). 

Next, we assume that $H$ is not a star graph. Let $e=x\sim y\in E_0=E(H)$ be an edge assuming $\min_{e\in E(H)}|S_1(e)_{(H)}|$. Assume without loss of generality that $x\in X$ and $y\in Y$. We apply the argument of \textbf{Case 2} in the proof of Theorem \ref{thm:non-regular and non-negative Lin--Lu--Yau curvature} to $e$. Then the property $k'(G)=\delta(G)-1$ forces the equalities in \eqref{eq:1.3_estimate_1} and \eqref{eq:1.3_estimate_2}. The first equality in \eqref{eq:1.3_estimate_1} tells that $\kappa_{LLY}(e)=0$. Applying Theorem \ref{thm:estimate of k_LLY in a min-cut}, we derive from the second equality that $x$ is adjacent to all other vertices in $A$ and $y$ is adjacent to all other vertices in $B$. The third equality in \eqref{eq:1.3_estimate_1} forces $d_x=d_y$ and, by Theorem \ref{thm:combinatorial}, (2)(a). The equalities in \eqref{eq:1.3_estimate_2} imply that $d_y=k'(G)+1$. Hence, we have $d_x=d_y=k'(G)+1$. 

If  $H\in \{K_{2,2}, \ nK_{1,1} \ \text{for } n \geq 2, \ 2K_{1,\frac{n}{2}} \ \text{for even } n > 2\}$, we observe that every edge $e\in E(H)$ has the same value $|S_1(e)_{(H)}|$. Hence, we can apply the above argument for each $e\in E(H)$. This proves (2)(b) and (2)(c). If $H\not\in \{K_{2,2}, \ nK_{1,1} \ \text{for } n \geq 2, \ 2K_{1,\frac{n}{2}} \ \text{for even } n > 2\}$, we deduce by (2)(a) that $H= ST_{\frac{n-1}{2},\frac{n-1}{2}}$ for odd $n \geq 3$.  Let $u\sim v, u\in A$ and $v\in B$ be the central edge of $ST_{\frac{n-1}{2},\frac{n-1}{2}}$. Then the edges assuming $\min_{e\in E(H)}|S_1(e)_{(H)}|$ are all edges in $E(H)$ except for $u\sim v$. Notice that those edges cover all vertices in $V(H)$. Applying the above argument for all those edges, we conclude (2)(b) and (2)(c).
\end{proof}

In the following lemma, we show that the local structure of the graph $H$ induced by a min-cut $E_0$ has further restrictions. 
\begin{lemma}\label{lemma:admissible_graph_H}
Let $G=(V,E)$ be a connected graph with non-negative Lin--Lu--Yau curvature and $k'(G)= \delta(G)-1$. Let $E_0\subset E$ be a min-cut of $G$ and $H$ be the edge-induced subgraph of $E_0$ in $G$. Then we have 
\[
H\in \{K_{2,2}, \ nK_{1,1} \ \text{for } n \geq 2, \ 2K_{1,2}, \ ST_{1,1}, \ K_{1,n} \ \text{for } n \geq 1\}.
\]
\end{lemma}
\begin{proof} By Lemma \ref{lemma:structure of G[]}, it is enough to show that $H$ cannot be $2K_{1,\frac{n}{2}}$ for even $n>4$ or $ST_{\frac{n-1}{2},\frac{n-1}{2}}$ for odd $n>3$.  Denote by $V(G)=X\sqcup Y$ the partition of $V(G)$ such that $E(X,Y)=E_0=E(H)$. Set $A:= X\cap V(H)$ and $B:=Y\cap V(H)$. 
We argue by contradiction. 

\textbf{Case 1:} Assume that $H =2K_{1,\frac{n}{2}}$ for even $n>4$. Then $n=k'(G)$. By Lemma \ref{lemma:structure of G[]}, we have either $G[A] = K_2, G[B] = K_n$ or $G[A]=G[B]=K_{\frac{n}{2}+1}$, see Figure \ref{fig:2_and_n} and Figure \ref{fig:n_halfplus1}.
\begin{figure}[H]
    \begin{minipage}{7cm}
    \centering
    
    \includegraphics[height=3cm,width=5cm]{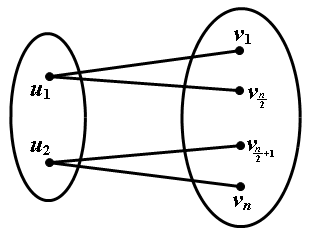}
    \caption{$G[A]\neq G[B]$}\label{fig:2_and_n}
    \end{minipage}
    \begin{minipage}{7cm}
    \centering
    
    \includegraphics[height=3cm,width=5cm]{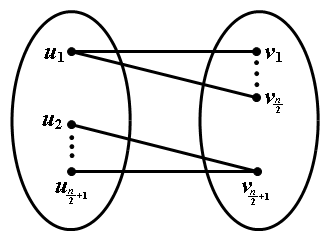}
    \caption{$G[A]=G[B]$}\label{fig:n_halfplus1}
    \end{minipage}
\end{figure}

We first consider the case that $G[A]=K_2$ and $G[B]=K_{n}$. Let us denote $A:= \{u_1, u_2\}$ and $B:=\{v_1,\ldots,v_n\}$ such that $u_1\sim v_{i}$ for any $i=1,\ldots \frac{n}{2}$ and $u_2\sim v_{j}$ for any $j=\frac{n}{2}+1,\ldots,n$. Applying Lemma \ref{lemma:structure of G[]}, we have $d_{u_i} = d_{v_j} = k'(G)+1=n+1$ for $i=1,2$ and $j = 1, \dots, n$, $|S_1(u_1)\setminus (A\cup B)|=\frac{n}{2}$ and $|S_1(v_1)\setminus(A\cup B)|=1$. Denote $S_1(u_1)  \setminus (A \cup B)=\{w_1, \dots, w_{\frac{n}{2}}\}$ and $S_1(v_1)  \setminus (A \cup B)=\{t\} $. Notice that 
\[ 
S_1(u_1) = \left\{u_2, v_1, \ldots, v_{\frac{n}{2}}, w_1, \ldots, w_{\frac{n}{2}}\right\}\ \text{and}\ S_1(v_1) = \left\{u_1, v_2, \ldots, v_r, t\right\}.
\]
Consider the Wasserstein distance $W(\mu_{u_1}, \mu_{v_1})$, where $\mu_{u_1}:=\mu_{u_1}^{\frac{1}{n+2}}, \mu_{v_1}:=\mu_{v_1}^{\frac{1}{n+2}}$ are the two uniform probability measures on $B_1(u_1), B_1(v_1)$, respectively. By (\ref{eq:pi_phi}) and (\ref{eq:W_transport_map}), an optimal transport plan can be induced by a bijective map $\phi:B_1(u_1)\to B_1(v_1)$ such that  $\phi(u)=u$ for any $u\in B_1(u_1)\cap B_1(v_1)$. If at least two of the vertices $w_1, \dots, w_{\frac{n}{2}}$, say $w_1$ and $w_2$, are not adjacent to $u_2$, then we have $d(w_1,\phi(w_1)) =d(w_2,\phi(w_2))=3$, and hence
\begin{equation*}
   \begin{split}
      (n+2)\cdot W(\mu_{u_1},\mu_{v_1}) &=d(u_2,\phi(u_2)) +\sum_{i=1}^2d(w_i,\phi(w_i))+\sum_{j=3}^{n/2}d(w_j,\phi(w_j))\\
        &\geq 1+6+2(n/2-2)=n+3.
   \end{split}
   \nonumber
\end{equation*}
 By (\ref{eq:K_LLY}), this implies the Lin--Lu--Yau curvature $\kappa_{LLY}(u_1,v_1)<0$, a contradiction. Therefore, at most one of the vertices $w_1, \dots, w_{\frac{n}{2}}$ is not adjacent to $u_2$. Without loss of generality, we assume $w_i\sim u_2$ for any $i=2,\ldots, \frac{n}{2}$. Since $d_{u_2}=k'(G)+1=n+1$ by Lemma \ref{lemma:structure of G[]}, the set $S_1(u_2) \setminus (\{w_2, \dots, w_{\frac{n}{2}}\} \cup A \cup B)$ contains only one vertex $w_0$. Notice that $w_0$ may coincide with $w_1$. Set $X':= X \setminus A$ and $Y':= Y \cup A$. Then $|E(X', Y')| = n$ and hence $E(X',Y')$ is a min-cut. 
 Let $H'$ be the edge-induced subgraph on $E(X',Y')$. Using Lemma \ref{lemma:structure of G[]} and recalling $n>4$ is even, we deduce that $H'\in \{nK_{1,1}, \ 2K_{1,\frac{n}{2}}, K_{1,n}\}.$
Observe that $A' :=X'\cap V(H') =\{w_1, \dots, w_{\frac{n}{2}}, w_0\}$ and $B' :=Y'\cap V(H')= \{u_1, u_2\}$. Then, $|A'|$ equals either $n/2$ or $n/2+1$ and $|B'|=2$. None of $nK_{1,1}, \ 2K_{1,\frac{n}{2}}, K_{1,n}$ satisfies this property. A contradiction. Therefore, the case $G[A]=K_2$ and $G[B]=K_n$ cannot happen. 

Next, we consider the case that $G[A]=G[B]=K_{\frac{n}{2}+1}$. Denote $A = \{u_1, \dots, u_{\frac{n}{2} + 1}\}$ and $B = \{v_1, \dots, v_{\frac{n}{2} + 1}\}$, where $u_1 \sim v_i$ for $i=1,\ldots,\frac{n}{2}$ and $u_j \sim v_{\frac{n}{2} + 1}$ for $j = 2, \dots, \frac{n}{2}+1$. By Lemma \ref{lemma:structure of G[]}, we deduce that $d_{u_i} = d_{v_i} = k'(G)+1=n+1$ for $i = 1, \dots, \frac{n}{2}+1$, $|S_1(u_1)\setminus(A\cup B)|=1$, and $|S_1(v_1)\setminus (A\cup B)|=\frac{n}{2}$. Let $S_1(u_1)  \setminus (A \cup B)=\{t\}$ and $S_1(v_1)  \setminus (A \cup B)=\{w_1, \dots, w_{\frac{n}{2}}\}$. Then
\[ 
S_1(u_1) = \left\{u_2, \dots, u_{\frac{n}{2} + 1}, v_1, \dots, v_{\frac{n}{2}}, t\right\}\ \text{and}\ S_1(v_1) = \left\{u_1, v_2, \dots, v_{\frac{n}{2} + 1}, w_1, \dots, w_{\frac{n}{2}}\right\}.
\]
Let $\phi:B_1(u_1)\to B_1(v_1)$ be a bijective map satisfying $\phi(u)=u$ for any $u\in B_1(u_1)\cap B_1(v_1)$, which induces an optimal transport plan between $\mu_{u_1}$ and $\mu_{v_1}$. If at least two of the vertices $w_1, \dots, w_{\frac{n}{2}}$, say $w_1$ and $w_2$, are not adjacent to $v_{\frac{n}{2}+1}$, then $d(\phi^{-1}(w_1),w_1)=d(\phi^{-1}(w_2),w_2)=3$. Thus,
\begin{align*}
      (n+2)\cdot W(\mu_{u_1},\mu_{v_1}) 
    =&d(\phi^{-1}(v_{\frac{n}{2}+1}),v_{\frac{n}{2}+1})+\sum_{i=1}^{n/2}d(\phi^{-1}(w_i),w_i)\geq 1+6+2(n/2-2)=n+3.
\end{align*}
By (\ref{eq:K_LLY}), we have $\kappa_{LLY}(u_1,v_1)<0$, a contradiction. Therefore, at most one of the vertices $w_1, \dots, w_{\frac{r}{2}}$ is not adjacent to $v_{\frac{r}{2}+1}$. Without loss of generality, we assume  let $w_i\sim v_{\frac{n}{2}+1}$ for $i=2,\dots,\frac{n}{2}$. Then
\[
d_{v_{\frac{n}{2}+1}}\geq \left|\left\{u_2,\dots,u_{\frac{n}{2}+1}\right\}\right|+\left|\left\{v_1,\dots,v_{\frac{n}{2}}\right\}\right|+\left|\left\{w_2,\dots,w_{\frac{n}{2}}\right\}\right|=\frac{3n}{2}-1.
\]
However, we have $d_{v_{\frac{n}{2}+1}}=k'(G)+1=n+1$ by Lemma \ref{lemma:structure of G[]}. Since $n>4$ is even, we have $\frac{3n}{2}-1>n+1$, a contradiction. Therefore, the case that $G[A]=G[B]=K_{\frac{n}{2}+1}$ cannot happen.

In conclusion, we have $H$ cannot be $2K_{1,\frac{n}{2}}$ for even $n>4$.

\textbf{Case 2:} Assume that $H = ST_{\frac{n-1}{2},\frac{n-1}{2}}$ for odd $n>3$. Then $n=k'(G)$.
Denote $A = \{u_1, \dots, u_{\frac{n + 1}{2}}\}$ and $B = \{v_1, \dots, v_{\frac{n + 1}{2}}\}$, where $u_1 \sim v_i$ and $u_i \sim v_{\frac{n + 1}{2}}$ for $i= 1, \dots, \frac{n + 1}{2}$. By Lemma \ref{lemma:structure of G[]}, we deduce that $d_{u_i} = d_{v_i} = k'(G)+1=n+1$ for $i = 1, \dots, \frac{n+1}{2}$, $|S_1(u_1)\setminus(A\cup B)|=1$ and $|S_1(v_1) \setminus (A \cup B)|=\frac{n+1}{2}$. Let $S_1(u_1)  \setminus (A \cup B)=\{t\}$ and $S_1(v_1) \setminus (A \cup B)=\{w_1, \dots, w_{\frac{n + 1}{2}}\}$. Then
\[ 
S_1(u_1) =\left \{u_2, \dots, u_{\frac{n + 1}{2}}, v_1, \dots, v_{\frac{n + 1}{2}}, t\right\}\ \text{and}\ S_1(v_1) = \left\{u_1, v_2, \dots, v_{\frac{n + 1}{2}}, w_1, \dots, w_{\frac{n + 1}{2}}\right\}.
\]
Let $\phi:B_1(u_1)\to B_1(v_1)$ be a bijective map satisfying $\phi(u)=u$ for any $u\in B_1(u_1)\cap B_1(v_1)$, which induces an optimal transport plan between $\mu_{u_1}$ and $\mu_{v_1}$. If at least two of the vertices $w_1, \dots, w_{\frac{n+1}{2}}$, say $w_1$ and $w_2$, are not adjacent to $v_{\frac{n+1}{2}}$, then $d(\phi^{-1}(w_1),w_1)=d(\phi^{-1}(w_2),w_2)=3$. Thus,
\begin{align*}
      (n+2)\cdot W(\mu_{u_1},\mu_{v_1})
        =&\sum_{i=1}^{(n+1)/2}d(\phi^{-1}(w_i),w_i)
        \geq 6+2((n+1)/2-2)=n+3.
\end{align*}
By (\ref{eq:K_LLY}), we have $\kappa_{LLY}(u_1,v_1)<0$, a contradiction. Therefore, at most one of the vertices $w_1, \dots, w_{\frac{n+1}{2}}$ is not adjacent to $v_{\frac{n+1}{2}}$. Without loss of generality, we assume that $w_i\sim v_{\frac{r+1}{2}}$ for any $i=2,\dots,\frac{n+1}{2}$. Then
\[
d_{v_{\frac{n+1}{2}}}\geq \left|\left\{u_1,\dots,u_{\frac{n+1}{2}}\right\}\right|+\left|\left\{v_1,\dots,v_{\frac{n-1}{2}}\right\}\right|+\left|\left\{w_2,\dots,w_{\frac{n+1}{2}}\right\}\right|=\frac{3n-1}{2}.
\]
Lemma \ref{lemma:structure of G[]} tells that $d_{v_{\frac{n+1}{2}}}=n+1$. However, $\frac{3n-1}{2}>n+1$ for odd $n>3$. This is a contradiction. Therefore, $H$ cannot be $ST_{\frac{n-1}{2},\frac{n-1}{2}}$ for any odd $n>3$. This finishes the proof. 
\end{proof}
Now we are prepared for the proof of Theorem \ref{thm:non-regular rigidity}.

\begin{proof}[Proof of Theorem \ref{thm:non-regular rigidity}]
Let $G$ be a connected graph with non-negative Lin--Lu--Yau curvature and $k'(G)=\delta(G)-1$.
Let $E_0\subset E$ be a min-cut of $G$. Denote by $H_0$ the edge-induced subgraph of $E_0$ in $G$, and $V(G)=X_0\sqcup Y_0$ be the partition of $V(G)$ such that $E(X_0,Y_0)=E_0=E(H_0)$. Set $A_0 = X_0\cap V(H_0)$, $B_0 =Y_0\cap V(H_0)$. Denote by $n:=k'(G)$ for convenience. 

By Lemma \ref{lemma:admissible_graph_H}, we have 
\begin{equation}\label{eq:H0}
    H_0\in \begin{cases}
        \{K_{1,n}, nK_{1,1}\}, & \text{if } n=1,\ 2\ \text{or}\ n\geq5;\\
        \{K_{1,3}, 3K_{1,1}, ST_{1,1}\}, & \text{if } n=3;\\
        \{K_{1,4}, 4K_{1,1}, K_{2,2}, 2K_{1,2}\}, & \text{if } n=4.
    \end{cases}
\end{equation}
Note that, if $H_0=2K_{1,2}$, we have either $|A_0|=|B_0|=3$ or $|A_0|, |B_0|\in \{2,4\}$.
We will discuss the case that $H_0=2K_{1,2}$ with $|A_0|=|B_0|=3$ separately, and deal with the remaining cases first.  

\vspace{.2cm}

\textbf{Case 1:} Either $H_0\neq 2K_{1,2}$, or $H_0=2K_{1,2}$ with $|A_0|, |B_0|\in \{2,4\}$.

Set $V_0=A_0$ and $V_1=B_0$. Then we have $V_0\cap V_1=\emptyset$. By Lemma \ref{lemma:structure of G[]}, it holds that $G[V_0]=K_{|V_0|}$ and $G[V_1]=K_{|V_1|}$. Let us define for $i\geq 2, i\in \mathbb{Z}$ that
\[V_i:=\{y\in V(G)\setminus (V_{i-2}\cup V_{i-1}): S_1(y)\cap V_{i-1}\neq \emptyset\},\]
and for $i\leq -1, i\in \mathbb{Z}$ that
\[V_i:=\{y\in V(G)\setminus (V_{i+1}\cup V_{i+2}): S_1(y)\cap V_{i+1}\neq \emptyset\}.\]

\begin{claim}\label{claim:induction-structure without G_4^*}
    For any $i\in \mathbb{Z}$, the edge set $E(V_{i-1}, V_{i})$ is a min-cut of $G$ and the subgraph $G[V_{i-1},V_i]$ coincides with the edge-induced graph of $E(V_{i-1},V_i)$.
\end{claim}
\begin{proof}
We prove by induction. The claim holds true for $i=1$. For any $k\geq 1$, suppose it is true for $i\leq k$. Let $V(G)=X\sqcup Y$ be the partition of $V(G)$ such that $V_{k-1}\subset X$, $V_k\subset Y$, and $E(X,Y)=E(V_{k-1},V_k)$. Then the sets $X':=X\cup V_k$ and $Y':=Y\setminus V_k$ provide a new partition of $V(G)$ such that $E(X',Y')=E(V_k,V_{k+1})$. That is, $E(V_k,V_{k+1})$ is a cut. Next, we show that $|E(V_k,V_{k+1})|=n:=k'(G)$ and hence $E(V_k,V_{k+1})$ is a min-cut. Observe that 
\begin{equation}\label{eq:cut}
    |E(V_k,V_{k+1})|=\sum_{x\in V_k}\left(d_x-|S_1(x)\cap(V_{k-1}\cup V_k)|\right). 
\end{equation}
By the induction assumption, $E(V_{k-1},V_k)$ is a min-cut of $G$ and the subgraph $G[V_{k-1},V_k]$ coincides with the edge-induced graph of $E(V_{k-1},V_k)$. By Lemma \ref{lemma:structure of G[]}, we have $G[V_k]=K_{|V_k|}$, and for any $x\in V_k$ that
\[d_x=\begin{cases}
    2n, & \text{if }|V_k|=1;\\
    n+1, & \text{if }|V_k|>1.
\end{cases}\]

If $|V_k|=1$, then we deduce from \eqref{eq:H0} that $E(V_{k-1},V_k)=K_{1,n}$. Denote $V_k=\{u\}$. Then $|S_1(u)\cap(V_{k-1}\cup V_k)|=|S_1(u)\cap V_{k-1}|=n$. By \eqref{eq:cut}, this yields that $|E(V_k,V_{k+1})|=2n-n=n$. By the definition of $V_{k+1}$, the subgraph $G[V_{k},V_{k+1}]$ coincides with the edge-induced graph of $E(V_{k},V_{k+1})$.

If, otherwise, $|V_k|>1$, then we have by \eqref{eq:H0} that $E(V_{k-1},V_k)\in\{nK_{1,1}, ST_{1,1}, K_{2,2}, 2K_{1,2}\}$. If $E(V_{i-1},V_i)=nK_{1,1}$, then $G[V_k]=K_{|V_k|}=K_n$. Moreover, for each $x\in V_k$, \[|S_1(x)\cap(V_{k-1}\cup V_k)|=1+(n-1)=n.\] Applying \eqref{eq:cut}, we obtain $|E(V_k,V_{k+1})|=n\times[(n+1)-n]=n$. If $E(V_{i-1},V_i)=ST_{1,1}$, then $n=3$ and $G[V_k]=K_2$. Denoting $V_k=\{u_1,u_2\}$, we have $d_{u_1}=d_{u_2}=4$, and \[|S_1(u_1)\cap(V_{k-1}\cup V_k)|+|S_1(u_2)\cap(V_{k-1}\cup V_k)|=3+2=5.\] Hence, by \eqref{eq:cut}, $|E(V_k,V_{k+1})|=d_{u_1}+d_{u_2}-5=3=n$. If $E(V_{k-1},V_k)=K_{2,2}$, then $n=4$ and $G[V_k]=K_2$. For every $x\in V_k$, \[|S_1(x)\cap(V_{k-1}\cup V_k)|=1+2=3.\] Therefore, $|E(V_k,V_{k+1})|=2\times(5-3)=4=n$. Finally, suppose that $E(V_{k-1},V_k)=2K_{1,2}$, then $n=4$. We claim that that $|V_{k}|=2$ or $4$. When $k=1$, this follows directly from the assumption in \textbf{Case 1}. For $k>1$, if this were not true, then $|V_{k-1}|=|V_k|=3$. By Lemma \ref{lemma:structure of G[]}, it follows that $G[V_{k-1}]=K_3$ and each vertex $u\in V_{k-1}$ has degree $5$. Hence, $|E(V_{k-2},V_{k-1]})|=15-4-6=5>4$, contradicting the assumption that $E(V_{k-2}, V_{k-1})$ is a min-cut. 

If $|V_k|=2$, then $G[V_k]=K_2$ and for any $x\in V_k$, we have $|S_1(x)\cap(V_{k-1}\cup V_k)|=2+1=3$. If $|V_k|=4$, then $G[V_k]=K_4$ and for each $x\in V_k$, $|S_1(x)\cap(V_{k-1}\cup V_k)|=1+3=4$. Consequently, in either situation, $|E(V_k,V_{k+1})|=4=n$.

The above analysis shows that $S_1(x)\cap V_{k+1}\neq\emptyset$ for any $x\in V_k$. Combined with the definition of $V_{k+1}$, this implies that the subgraph $G[V_{k},V_{k+1}]$ coincides with the edge-induced graph of $E(V_{k-1},V_k)$. 

By the above induction argument, we show Claim \ref{claim:induction-structure without G_4^*} for any $i\geq 1$. The case $i\leq 1$ follows from a similar induction argument. This completes the proof. 
\end{proof}

By definition, for every $i\in \mathbb{Z}$, we have $V_i\cap V_{i+1}=\emptyset$. Moreover, since $E(V_i,V_{i+1})$ is a min-cut for each $i\in \mathbb{Z}$, it follows that $V_i\cap V_j=\emptyset$ and $E(V_i,V_j)=\emptyset$ whenever $|i-j|>1$. Because $G$ is connected, it follows that $V(G)=\sqcup_{i\in \mathbb{Z}}V_i$ with $E(V_i,V_j)=\emptyset$ whenever $|i-j|>1$.

We examine the following three subcases according to the value of $n=k'(G)$.

\textbf{Subcaes 1.1:} $n=1, 2$ or $n\geq5$.

By Lemma~\ref{lemma:structure of G[]}, Lemma~\ref{lemma:admissible_graph_H}, and Claim~\ref{claim:induction-structure without G_4^*}, for any $i\in \mathbb{Z}$, we have $G[V_i]\in\{K_1,K_n\}$ and $G[V_i,V_{i+1}]\in\{K_{1,n},nK_{1,1}\}$. Therefore, $G\in \mathcal{G}_n$.

\textbf{Subcaes 1.2:} $n=3$.

If $E(V_i,V_{i+1})\neq ST_{1,1}$ for any $i\in\mathbb{Z}$, then, by Lemma~\ref{lemma:structure of G[]}, Lemma~\ref{lemma:admissible_graph_H}, and Claim~\ref{claim:induction-structure without G_4^*}, we have $G[V_i]\in\{K_1,K_3\}$ and $G[V_i,V_{i+1}]\in\{K_{1,3},3K_{1,1}\}$ for any $i\in\mathbb{Z}$. Therefore, $G\in \mathcal{G}_3$.

If, otherwise, there exists an $i_0\in\mathbb{Z}$ such that $E(V_{i_0},V_{i_0+1})=ST_{1,1}$, then we claim that $E(V_{i-1}, V_{i})=ST_{1,1}$ for any $i\in \mathbb{Z}$. Without loss of generality, suppose $E(V_0,V_1)=ST_{1,1}$. We argue by induction. The claim holds true for $i=1$. Suppose it is true for $i=k\geq 1$. By Lemma~\ref{lemma:structure of G[]} and Lemma~\ref{lemma:admissible_graph_H}, we have $E(V_{k}, V_{k+1})\in\{K_{1,3},3K_{1,1},ST_{1,1}\}$. Since $|V_k|=2$, $E(V_k, V_{k+1})=ST_{1,1}$. This shows the claim for any $i\geq 1$. The case $i\leq 1$ follows by a similar induction argument.

Furthermore, applying Lemma~\ref{lemma:structure of G[]} tells that $G$ is $4-$regular and $G[V_i]=K_2$ for any $i\in \mathbb{Z}$. Therefore, $G=G_3^*$.

\textbf{Subcaes 1.3:} $n=4$.

By Lemma~\ref{lemma:structure of G[]}, Lemma~\ref{lemma:admissible_graph_H}, and Claim~\ref{claim:induction-structure without G_4^*}, we have $G[V_i]\in\{K_1,K_2,K_4\}$ for any $i\in\mathbb{Z}$, and $G[V_i,V_{i+1}]\in\{K_{1,4},4K_{1,1},K_{2,2},2K_{1,2}\}$ for any $i\in\mathbb{Z}$. Suppose that there exists $i\in \mathbb{Z}$ such that $(|V_{i-1}|,|V_i|,|V_{i+1}|)= (4,2,4)$, then the local structure of the graph is as depicted in Figure \ref{fig:negative_curvature}. It is direct to check by definition that the Lin--Lu--Yau curvature of any edge in $E(V_{i-1},V_i)\sqcup E(V_i,V_{i+1})$ is equal to $-0.2$. This contradicts to the assumption that $G$ has non-negative Lin--Lu--Yau curvature. Therefore, $G\in \Tilde{\mathcal{G}}_4$.

To conclude the proof for \textbf{Case 1}, we mention that all graphs in $\mathcal{G}_n, n\in \mathbb{Z}$ and $\Tilde{\mathcal{G}}_4$ and the graph $G_3^*$ have non-negative Lin--Lu--Yau curvature. See Appendix \ref{subsection:non-negative LLY curvature}.

\vspace{.2cm}

\textbf{Case 2:}  $H_0=2K_{1,2}$ with $|A_0|=|B_0|=3$.

In this case, we have $n=|H_0|=4$, $\delta(G)=n+1=5$. Let $A_0 = \{u_1, u_2, u_3\}$ and $B_0 = \{v_1, v_2, v_3\}$, where $u_1 \sim v_i$ and $u_{1+i} \sim v_3$ for $i = 1, 2$. By Lemma \ref{lemma:structure of G[]}, we have $G[A_0]=G[B_0]=K_3$, and $d_{u_i} = d_{v_i} =n+1= 5$ for $i = 1, 2, 3$. Therefore, $|S_1(u_1) \setminus (A_0 \cup B_0)|=|S_1(v_3) \setminus (A_0 \cup B_0)|=1$. Let $S_1(u_1) \setminus (A_0 \cup B_0)=\{u_4\}$ and $S_1(v_3) \setminus (A_0 \cup B_0)=\{v_4\}$. Set $V_0=A_0\cup\{u_4\}$ and $V_1=B_0\cup\{v_4\}$.

\begin{claim}\label{claim:G[V_0]=K_4 and d_u_4=5}
$G[V_0]=G[V_1]=K_4$ and $d_{u_4}=d_{v_4}=5$.
\end{claim}

\begin{proof}
Let us consider the Lin--Lu--Yau curvature of the edge $e_1:=u_1\sim v_1$. Notice that \[S_1(u_1) = \{u_2, u_3, u_4, v_1, v_2\}.\]
Observe that $|S_1(v_1) \setminus (A_0 \cup B_0)|=5-3=2$. Denote $S_1(v_1) \setminus (A_0 \cup B_0)=\{w_1, w_2\} $, see Figure~\ref{fig:add w_1 w_2}. Then
\[ 
\ S_1(v_1) = \{u_1, v_2, v_3, w_1, w_2\}.
\]
\begin{figure}[ht]  
    \centering  
    \includegraphics[width=0.5\textwidth]{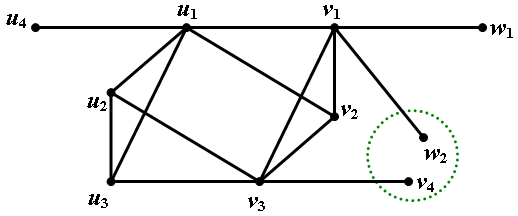}
    \caption{Local structure of $H_0$.}
    \label{fig:add w_1 w_2} 
\end{figure}

Recall from \eqref{eq:pi_phi} that there exists a bijective map $\phi:B_1(u_1)\longrightarrow B_1(v_1)$ such that  $\phi(u)=u$ for any $u\in B_1(u_1)\cap B_1(v_1)$ and 
\[ 6\cdot W(\mu_{u_1},\mu_{v_1}) = \sum_{u\in B_1(u_1)\setminus B_1(v_1)}d(u,\phi(u)).\]
If $v_3\not\sim w_1$ and $v_3\not\sim w_2$, then $d(\phi^{-1}(w_1),w_1)=d(\phi^{-1}(w_2),w_2)=3$. Thus,
\begin{equation*}
      6\cdot W(\mu_{u_1},\mu_{v_1})=d(\phi^{-1}(w_1))+d(\phi^{-1}(w_2),w_2)+d(\phi^{-1}(v_3),v_3)\geq 3+3+1 =7.
\end{equation*}
By (\ref{eq:K_LLY}), the above inequality implies that $\kappa_{LLY}(e_1)<0$, a contradiction. Therefore, we have $v_3\sim w_1$ or $v_3\sim w_2$. Without loss of generality, let $v_3\sim w_2$. Since $d_{v_3} = 5$, we have $v_4=w_2$ and hence $v_4\sim v_1$. By a similar argument for the edge $u_1\sim v_2$, we obtain $v_4\sim v_2$. Hence, $G[V_1]=K_4$.

Similarly, we derive $u_4 \sim u_2$, $u_4 \sim u_3$, and hence $G[V_0]=K_4$. Next, we show $d_{u_4}=d_{v_4}=5$. Since $\delta(G) = 5$, $d_{u_4}, d_{v_4} \geq 5$.
\begin{figure}[ht]  
    \centering  
    \includegraphics[width=0.5\textwidth]{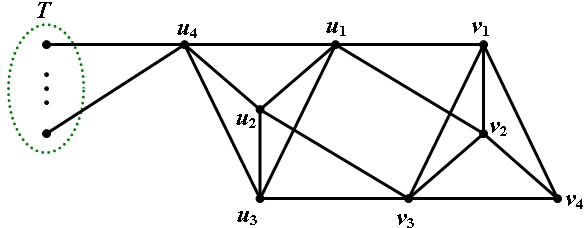}
    \caption{Local structure of $u_4$.}
    \label{fig:suppose degree of u_4 geq 6} 
\end{figure}
Suppose $d_{u_4} \geq 6$.
We consider the Lin--Lu--Yau curvature of the edge $e_* = u_1 \sim u_4$.  Set $T:=S_1(u_4) \setminus \{u_1, u_2, u_3\}$, see Figure~\ref{fig:suppose degree of u_4 geq 6}. By the Kantorovich duality \eqref{eq:Kantorovich}, we have
\begin{equation*}
  \begin{aligned}
    W\left(\mu_{u_1}^{\frac{1}{d_{u_1}+1}},\mu_{u_4}^{\frac{1}{d_{u_1}+1}}\right)
    \geq \sum_{v \in V} d(v,T) \left(\mu_{u_1}^{\frac{1}{d_{u_1}+1}}(v) - \mu_{u_4}^{\frac{1}{d_{u_1}+1}}(v)\right).
  \end{aligned}
\end{equation*}
Since $d(v,T)=0$ for any $v \in T$, we have
\begin{equation*}
  \begin{aligned}
    W\left(\mu_{u_1}^{\frac{1}{d_{u_1}+1}},\mu_{u_4}^{\frac{1}{d_{u_1}+1}}\right)\geq& \sum_{v \in \{u_1,u_2,u_3\}} d(v,T) \left(\mu_{u_1}^{\frac{1}{d_{u_1}+1}}(v)-\mu_{u_4}^{\frac{1}{d_{u_1}+1}}(v)\right)+\sum_{v \in \{v_1,v_2\}} d(v,T) \mu_{u_1}^{\frac{1}{d_{u_1}+1}}(v) \\
    \geq& 2\times \left(\frac{1}{6}-\frac{5}{36}\right)+2\times1\times\left(\frac{1}{6}-\frac{5}{36}\right)+2\times 3\times\frac{1}{6}>1.
  \end{aligned}
\end{equation*}
By (\ref{eq:K_LLY}), the above inequality implies that $\kappa_{LLY}(e_*) < 0$, a contradiction. Thus, we derive $d_{u_4} = 5$. Similarly, we have $d_{v_4} = 5$.

\end{proof}

Let us define, for $i\geq 2, i\in \mathbb{Z}$, that $V_i:=R_i^1\sqcup R_i^2$, where
\[R_i^1:=\{y\in V(G)\setminus (V_{i-2}\cup V_{i-1}): S_1(y)\cap V_{i-1}\neq \emptyset\},\ \ R_i^2:=\{y\in V(G)\setminus V_{i-1}: R_i^1\subset S_1(y)\};\]
and for $i\leq -1, i\in \mathbb{Z}$, that $V_i=L_i^1\sqcup L_i^2$, where
\[L_i^1:=\{y\in V(G)\setminus (V_{i+1}\cup V_{i+2}): S_1(y)\cap V_{i+1}\neq \emptyset\},\ \ L_i^2:=\{y\in V(G)\setminus V_{i+1}: L_i^1\subset S_1(y)\}.\]
Observe from the above definition that $V_i\cap V_{i+1}=\emptyset$ for any $i\in \mathbb{Z}$.
For any $i\in \mathbb{Z}$, we denote by $H_i$ the edge-induced subgraph of $E(V_{i},V_{i+1})$ in $G$. Set $A_i:=V_{i}\cap V(H_i)$ and $B_i:=V_{i+1}\cap V(H_i)$.

Next, we show for each $i\in\mathbb{Z}$, the following properties hold true:
\begin{enumerate}
\item $E(V_{i-1}, V_{i})$ is a min-cut of $G$, with $H_{i-1}=2K_{1,2}$ and $|A_{i-1}|=|B_{i-1}|=3$;
\item $d_x=5$ for any $x\in V_{i-1}\cup V_i$;
\item $G[V_i]=K_4$.
\end{enumerate}

We again argue by induction. By Claim~\ref{claim:G[V_0]=K_4 and d_u_4=5}, The above properties hold true for $i=1$. Suppose they are true for $i=k\geq 1$. Let $V(G)=X\sqcup Y$ be the partition of $V(G)$ such that $V_{k-1}\subset X$, $V_k\subset Y$, and $E(X,Y)=E(V_{k-1},V_k)$. Since $E(V_{k-1},V_k)$ is a cut of $G$, we have $V_{k+1}\cap X=\emptyset$. Then the sets $X':=X\cup V_k$ and $Y':=Y\setminus V_k$ provide a new partition of $V(G)$ such that $E(X',Y')=E(V_k,V_{k+1})$. Thus, $E(V_k,V_{k+1})$ is a cut. Next, we show it is a min-cut.

Let $V_j:=\{u_j^1,u_j^2,u_j^3,u_j^4\}$ for $j=k-1, k$. Without loss of generality, we assume that \[u_{k-1}^1 \sim u_k^j,\text{ and}\ u_{k-1}^{1+j} \sim u_k^3,\,\ \text{for}\ j = 1, 2.\] By the induction assumption, $d_x=5$ for any $x\in V_k$. Set $S_1(u_k^4)\setminus\{u_k^1,u_k^2,u_k^3\}:=\{u_{k+1}^1,u_{k+1}^2\}$, $S_1(u_k^2)\setminus\{u_{k-1}^1,u_k^1,u_k^3,u_k^4\}:=\{u_{k+1}^3\}$, and $S_1(u_k^1)\setminus\{u_{k-1}^1,u_k^2,u_k^3,u_k^4\}:=\{\Tilde{u}_{k+1}^3\}$. Notice further that all neighbors of $u_k^3$ lie in $V_{k-1}\sqcup V_k$.
Thus, $|E(V_k,V_{k+1})|=4$ and $E(V_k,V_{k+1})$ is a min-cut. 
Recall that $A_{k}=V(H_{k})\cap V_k=\{u_k^1,u_k^2,u_k^4\}$. Then, by Lemma \ref{lemma:admissible_graph_H}, we have $H_{k}=2K_{1,2}$ with $|A_{k}|=|B_{k}|=3$. This proves the property $(1)$ holds for $i=k+1$.

Moreover, we conclude that $u_{k+1}^3=\Tilde{u}_{k+1}^3$. Thus, $B_{k}=V(H_{k})\cap V_{k+1}=\{u_{k+1}^1,u_{k+1}^2,u_{k+1}^3\}$. By Lemma \ref{lemma:structure of G[]}, we have $G[B_{k}]=K_3$, and $d_{u_{k+1}^j} = 5$ for $j = 1, 2, 3$. Let $S_1(u_{k+1}^3) \setminus \{u_{k}^1,u_{k}^2,u_{k+1}^1,u_{k+1}^2\}=\{u_{k+1}^4\}$. By Claim~\ref{claim:G[V_0]=K_4 and d_u_4=5}, we have $G[\{u_{k+1}^1,u_{k+1}^2,u_{k+1}^3,u_{k+1}^4\}]=K_4$ and $d_{u_{k+1}^4}=5$. According to the definition of $V_{k+1}$, we have $V_{k+1}=\{u_{k+1}^1,u_{k+1}^2,u_{k+1}^3,u_{k+1}^4\}$. Thus, $G[V_{k+1}]=K_4$, and $d_x=5$ for any $x\in V_{k}\cup V_{k+1}$. That is, the properties $(2)$ and $(3)$ hold for $i=k+1$. By our induction argument, we show the properties $(1)-(3)$ hold for $i\geq 1$. The case $i\leq 1$ follows by a similar induction argument.

Recall that $V_i\cap V_{i+1}=\emptyset$ for any $i\in \mathbb{Z}$. Since $E(V_i,V_{i+1})$ is a min-cut for each $i\in \mathbb{Z}$, we observe that $V_i\cap V_j=\emptyset$ and $E(V_i, V_j)=\emptyset$ whenever $|i-j|>1$. Since $G$ is connected, we have $V(G)=\sqcup_{i\in \mathbb{Z}}V_i$.
Therefore, $G=G_4^*$. Moreover, notice that the graph $G_4^*$ has non-negative Lin--Lu--Yau curvature, see Appendix \ref{subsection:non-negative LLY curvature}.

Combining the above discussions in \textbf{Case 1} and \textbf{Case 2}, we finish the proof of Theorem \ref{thm:non-regular rigidity}.
\end{proof}

\appendix
\section{The Lin--Lu--Yau curvature of graphs listed in Theorem \ref{thm:non-regular rigidity}}\label{subsection:non-negative LLY curvature}
In this Appendix, we present the Lin--Lu--Yau curvature values for edges of graphs $G$ in \[\left(\cup_{n\in\mathbb{Z}_{>0}}\mathcal{G}_n\right)\cup\Tilde{\mathcal{G}}_4\cup\{G_3^*,G_4^*\}.\]
The Lin--Lu--Yau curvature of an edge is determined using \eqref{eq:K_LLY}. Notice that the Wasserstein distance $W$ can be obtained by using  Definition \ref{definition:Wasserstein} together with the Kantorovich duality formula \eqref{eq:Kantorovich}.
More precisely, Definition \ref{definition:Wasserstein} provides an upper bound for $W$, while the Kantorovich duality \eqref{eq:Kantorovich} gives a lower bound. When these two bounds agree, the exact value of $W$ follows. One may also use the \emph{graph curvature calculator} (https://www.mas.ncl.ac.uk/graph-curvature/) to compute Lin--Lu--Yau curvature values, see \cite{Cush22}. 

\textbf{Case 1:} $G=G_3^*$. Let $V(G)=\sqcup_{i\in\mathbb{Z}}V_i$ be the partition in Definition \ref{def:The graphs G_3^*}. Then an edge $e$ of $G$ lies either in $G[V_i]$ or in $E(V_i, V_{i+1})$ for some integer $i$. If $e\in E(G[V_i])$, then $\kappa_{LLY}(e)=0.5$, see Figure~\ref{fig:G_3^ e in E(G[V_i])}. If, otherwise, $e\in E(V_i, V_{i+1})$, then we have $\kappa_{LLY}(e)\in\{0,0.5\}$, see Figure~\ref{fig:G_3^ e in E(V_i,V_{i+1})}. 
\begin{figure}[H]
    \begin{minipage}{7.4cm}
    \centering
    \includegraphics[height=3cm,width=6cm]{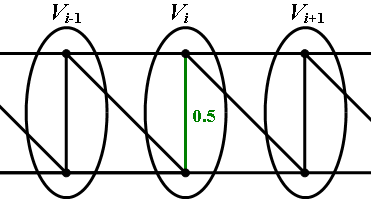}
    \caption{$\kappa_{LLY}(e)=0.5$}
    \label{fig:G_3^ e in E(G[V_i])}
    \end{minipage}
    \begin{minipage}{7.4cm}
    \centering
    \includegraphics[height=3cm,width=4cm]{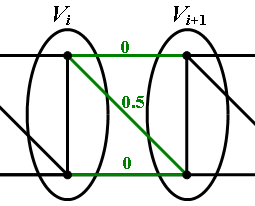}
    \caption{$\kappa_{LLY}(e)\in\{0,0.5\}$}
    \label{fig:G_3^ e in E(V_i,V_{i+1})}
    \end{minipage}
\end{figure}

\textbf{Case 2:} $G=G_4^*$. Let $V(G)=\sqcup_{i\in\mathbb{Z}}V_i$ be the partition in Definition \ref{def:The graphs G_4^*}. Then an edge $e$ of $G$ lies either in $G[V_i]$ or in $E(V_i, V_{i+1})$ for some integer $i$. If $e\in E(G[V_i])$, then $\kappa_{LLY}(e)\in\{0,0.4,1.2\}$, see Figure~\ref{fig:G_4^ e in E(G[V_i])}. If, otherwise, $e\in E(V_i, V_{i+1})$, then we have $\kappa_{LLY}(e)=0$, see Figure~\ref{fig:G_4^ e in E(V_i,V_{i+1})}.
\begin{figure}[H]
    \begin{minipage}{8cm}
    \centering
    \includegraphics[height=3.5cm,width=7.5cm]{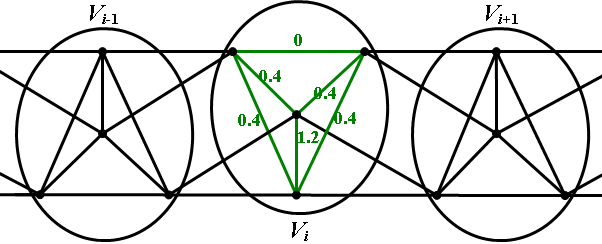}
    \caption{$\kappa_{LLY}(e)\in\{0, 0.4, 1.2\}$}
    \label{fig:G_4^ e in E(G[V_i])}
    \end{minipage}
    \begin{minipage}{7cm}
    \centering
    \includegraphics[height=3.5cm,width=4.9cm]{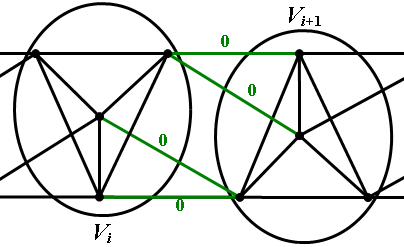}
    \caption{$\kappa_{LLY}(e)=0$}
    \label{fig:G_4^ e in E(V_i,V_{i+1})}
    \end{minipage}
\end{figure}
\textbf{Case 3:} $G\in\mathcal{G}_n$ for some positive integer $n$. Let $V(G)=\sqcup_{i\in\mathbb{Z}}V_i$ be the partition in Definition \ref{def:The graph class mathcal G_n}. Then an edge $e$ of $G$ lies either in $G[V_i]$ or in $E(V_i, V_{i+1})$ for some integer $i$. If $e\in E(G[V_i])$, then $G[V_i]$ is not $K_1$ but $K_n$, and the local structure around $e$ can only be one of the following three subcases: (i) $G[V_{i-1}]=G[V_{i+1}]=K_1$; (ii) $\{G[V_{i-1}],G[V_{i+1}]\}=\{K_1,K_n\}$; (iii) $G[V_{i-1}]=G[V_{i+1}]=K_n$; see Figure~\ref{fig:G_n two point e in E(G[V_i])},~\ref{fig:G_n one point e in E(G[V_i])}, and~\ref{fig:G_n none point e in E(G[V_i])}. The curvature $\kappa_{LLY}(e)$ equals $\frac{n+2}{n+1}$, $1$, $\frac{n}{n+1}$, respectively.
\begin{figure}[H]
    \begin{minipage}{7.1cm}
    \centering
    \includegraphics[height=3.5cm,width=7.1cm]{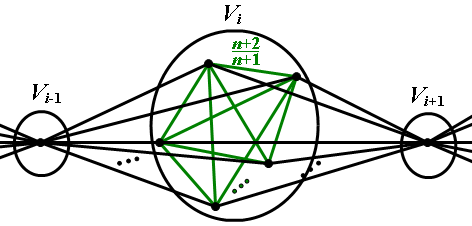}
    \caption{$\kappa_{LLY}(e)=\frac{n+2}{n+1}$}
    \label{fig:G_n two point e in E(G[V_i])}
    \end{minipage}
    \begin{minipage}{7.1cm}
    \centering
    \includegraphics[height=3.5cm,width=7.1cm]{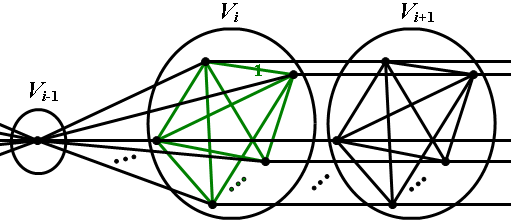}
    \caption{$\kappa_{LLY}(e)=1$}
    \label{fig:G_n one point e in E(G[V_i])}
    \end{minipage}
\end{figure}
\begin{figure}[H]
    \begin{minipage}{8cm}
    \centering
    \includegraphics[height=3.5cm,width=8cm]{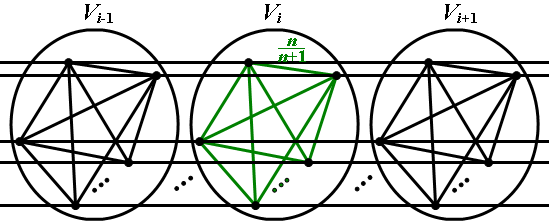}
    \caption{$\kappa_{LLY}(e)=\frac{n}{n+1}$}
    \label{fig:G_n none point e in E(G[V_i])}
    \end{minipage}
\end{figure}
If, otherwise, $e\in E(V_i, V_{i+1})$, then we have $\kappa_{LLY}(e)=0$. In fact, the edge-induced subgraph of $E(V_i, V_{i+1})$ is isomorphic to either $K_{1,n}$ or $nK_{1,1}$, see Figure~\ref{fig:G_n one point e in E(V_i, V_{i+1})} and~\ref{fig:G_n none point e in E(V_i, V_{i+1})}.
\begin{figure}[H]
    \begin{minipage}{7.1cm}
    \centering
    \includegraphics[height=3.5cm,width=5cm]{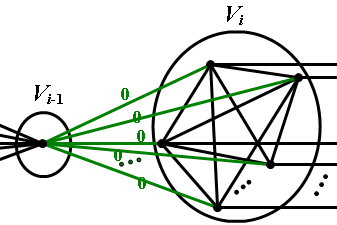}
    \caption{$\kappa_{LLY}(e)=0$}
    \label{fig:G_n one point e in E(V_i, V_{i+1})}
    \end{minipage}
    \begin{minipage}{7.1cm}
    \centering
    \includegraphics[height=3.5cm,width=5.5cm]{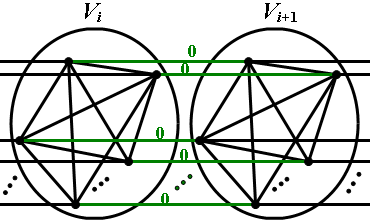}
    \caption{$\kappa_{LLY}(e)=0$}
    \label{fig:G_n none point e in E(V_i, V_{i+1})}
    \end{minipage}
\end{figure}
\textbf{Case 4:} $G\in\Tilde{\mathcal{G}}_4$. Let $V(G)=\sqcup_{i\in\mathbb{Z}}V_i$ be the partition in Definition \ref{def:The graph set Tilde mathcal G_4}. Then an edge $e$ of $G$ lies either in $G[V_i]$ or in $E(V_i, V_{i+1})$ for some integer $i$. If $e\in E(G[V_i])$, then $G[V_i]\in\{K_2,K_4\}$. If $G[V_i]=K_4$, then the local structure near $e$ can only be one of the following six subcases:
\begin{multicols}{2} 
\begin{enumerate}
\item $G[V_{i-1}]=G[V_{i+1}]=K_1$;
\item $G[V_{i-1}]=G[V_{i+1}]=K_2$;
\item $G[V_{i-1}]=G[V_{i+1}]=K_4$;
\item $\{G[V_{i-1}],G[V_{i+1}]\}=\{K_1,K_2\}$;
\item $\{G[V_{i-1}],G[V_{i+1}]\}=\{K_1,K_4\}$;
\item $\{G[V_{i-1}],G[V_{i+1}]\}=\{K_2,K_4\}$.
\end{enumerate}    
\end{multicols}
The curvature value of $e$ is depicted in Figure \ref{fig:K_4 with two point e in E(G[V_i])}--\ref{fig:K_4 with K_2 K_4 e in E(G[V_i])}.

\begin{figure}[H]
    \begin{minipage}{7.1cm}
    \centering
    \includegraphics[height=3.5cm,width=6.5cm]{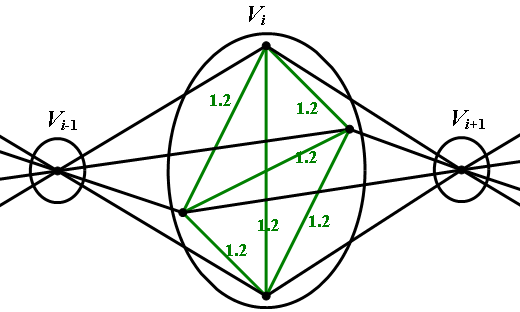}
    \caption{$\kappa_{LLY}(e)=1.2$}
    \label{fig:K_4 with two point e in E(G[V_i])}
    \end{minipage}
    \begin{minipage}{7.8cm}
    \centering
    \includegraphics[height=3.5cm,width=5.3cm]{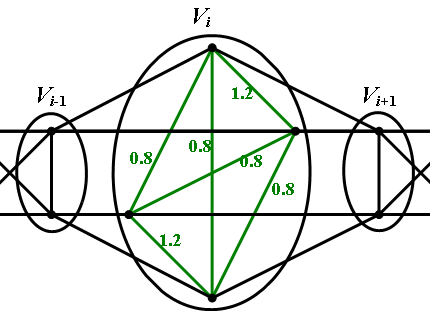}
    \caption{$\kappa_{LLY}(e)\in\{0.8,1.2\}$}
    \label{fig:K_4 with two K_2 e in E(G[V_i])}
    \end{minipage}
\end{figure}

\begin{figure}[H]
    \begin{minipage}{7cm}
    \centering
    \includegraphics[height=2.9cm,width=6.65cm]{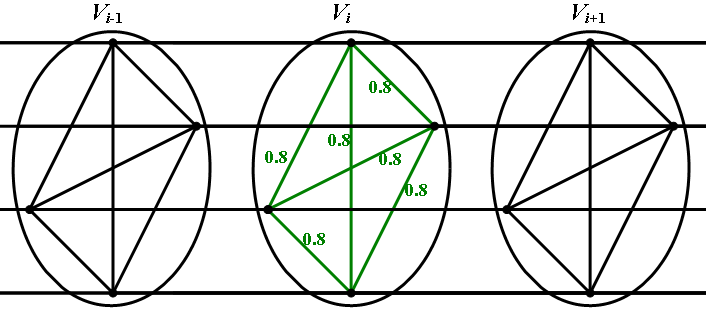}
    \caption{$\kappa_{LLY}(e)=0.8$}
    \label{fig:K_4 with two K_4 e in E(G[V_i])}
    \end{minipage}
    \begin{minipage}{7.5cm}
    \centering
    \includegraphics[height=2.9cm,width=4.4cm]{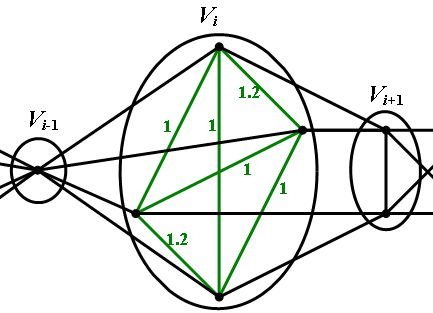}
    \caption{$\kappa_{LLY}(e)\in\{1,1.2\}$}
    \label{fig:K_4 with K_1 K_2 e in E(G[V_i])}
    \end{minipage}
\end{figure}

\begin{figure}[H]
    \begin{minipage}{7.1cm}
    \centering
    \includegraphics[height=3.5cm,width=6.5cm]{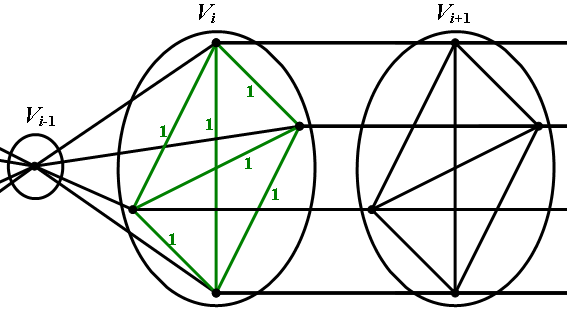}
    \caption{$\kappa_{LLY}(e)=1$}
    \label{fig:K_4 with K_1 K_4 e in E(G[V_i])}
    \end{minipage}
    \begin{minipage}{7.4cm}
    \centering
    \includegraphics[height=3.5cm,width=6.5cm]{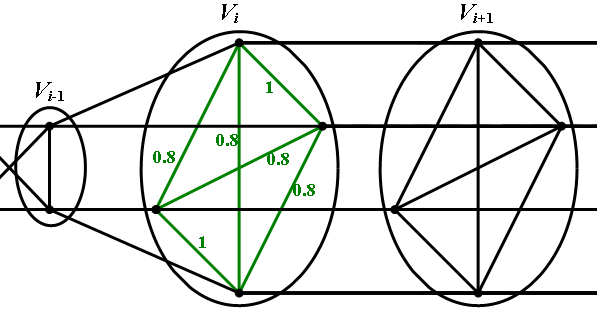}
    \caption{$\kappa_{LLY}(e)\in\{0.8,1\}$}
    \label{fig:K_4 with K_2 K_4 e in E(G[V_i])}
    \end{minipage}
\end{figure}

If $G[V_i]=K_2$, then the local structure around $e$ can only be one of the following two subcases: (a) $G[V_{i-1}]=G[V_{i+1}]=K_2$; (b) $\{G[V_{i-1}],G[V_{i+1}]\}=\{K_2,K_4\}$.
See Figures \ref{fig:K_2 with two K_2 e in E(G[V_i])} and Figure \ref{fig:K_2 with one K_2 e in E(G[V_i])} for the curvature values.

\begin{figure}[H]
    \begin{minipage}{7cm}
    \centering
    \includegraphics[height=1.7cm,width=3.6cm]{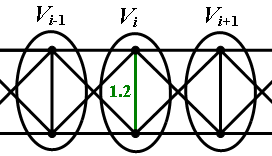}
    \caption{$\kappa_{LLY}(e)=1.2$}
    \label{fig:K_2 with two K_2 e in E(G[V_i])}
    \end{minipage}
    \begin{minipage}{7cm}
    \centering
    \includegraphics[height=3.5cm,width=5.2cm]{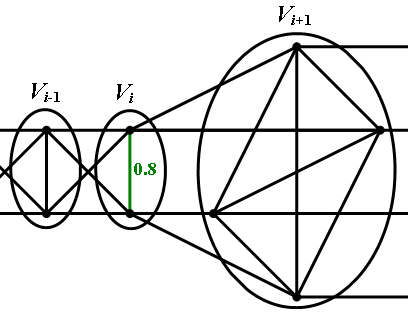}
    \caption{$\kappa_{LLY}(e)=0.8$}
    \label{fig:K_2 with one K_2 e in E(G[V_i])}
    \end{minipage}
\end{figure}
If $e\in E(V_i, V_{i+1})$, then we have $\kappa_{LLY}(e)=0$. In fact, the edge-induced subgraph of $E(V_i, V_{i+1})$ is isomorphic to one of $K_{1,4}$, $4K_{1,1}
$, and $2K_{1,2}$, see Figure~\ref{fig:K_1,4 e in E(V_i, V_{i+1})},~\ref{fig:4K_1,1 e in E(V_i, V_{i+1})}, and~\ref{fig:2K_1,2 e in E(V_i, V_{i+1})}. Notice that in Figure \ref{fig:2K_1,2 e in E(V_i, V_{i+1})}, the induced subgraph $G[V_{i-1}]$ can only be $K_2$ due to Condition (iv) in Definition~\ref{def:The graph set Tilde mathcal G_4}.
\begin{figure}[H]
    \begin{minipage}{4.9cm}
    \centering
    \includegraphics[height=3cm,width=4cm]{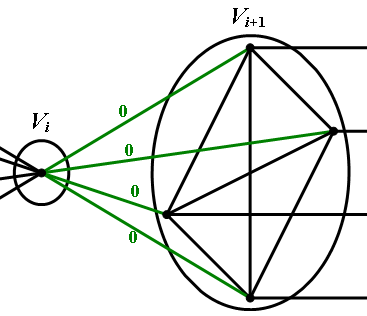}
    \caption{$\kappa_{LLY}(e)\text{ = 0}$}
    \label{fig:K_1,4 e in E(V_i, V_{i+1})}
    \end{minipage}
    \begin{minipage}{5cm}
    \centering
    \includegraphics[height=3cm,width=5cm]{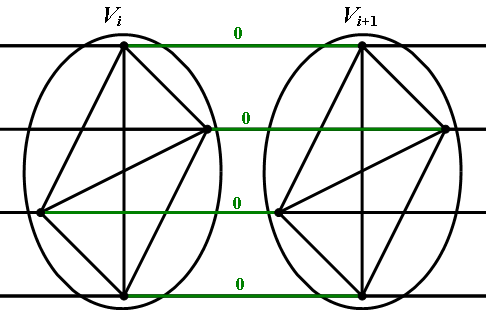}
    \caption{$\kappa_{LLY}(e)\text{ = 0}$}
    \label{fig:4K_1,1 e in E(V_i, V_{i+1})}
    \end{minipage}
    \begin{minipage}{4.9cm}
    \centering
    \includegraphics[height=3cm,width=4cm]{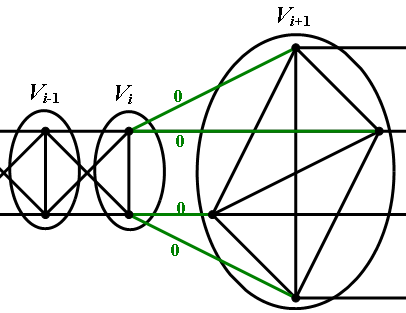}
    \caption{$\kappa_{LLY}(e)\text{ = 0}$}
    \label{fig:2K_1,2 e in E(V_i, V_{i+1})}
    \end{minipage}
\end{figure}
\section*{Acknowledgements}
This work is supported by the National Key R \& D Program of China 2023YFA1010200 and the National Natural Science Foundation of China No. 12431004. We are very grateful to Kaizhe Chen for many useful discussions, especially for providing one of the sharp graphs, $G_3^*$.



\begin{thebibliography}{99}
\bibitem{BJL12} F. Bauer, J. Jost, S. Liu,
Ollivier-Ricci curvature and the spectrum of the normalized graph Laplace operator,
Math. Res. Lett. 19 (2012), no. 6, 1185-1205.

\bibitem{BM15} B. B. Bhattacharya, S. Mukherjee, Exact and asymptotic results on coarse Ricci curvature of graphs, Discrete Math. 338 (2015), no. 1, 23-42.

\bibitem{BCLMP18} D. Bourne, D. Cushing, S. Liu, F. M\"unch, N. Peyerimhoff, Ollivier-Ricci idleness functions of graphs, SIAM J. Discrete Math. 32 (2018), no. 2, 1408-1424.


\bibitem{BrHa05} A. E. Brouwer and W. H. Haemers, Eigenvalues and perfect matchings, Linear Algebr. Appl., 395 (2005), pp. 155-162.
\bibitem{BrMe85} A. E. Brouwer and D. M. Menser, The connectivity for strongly regular graphs, European J. Combin. 6 (1985), no. 3, 215-216.

\bibitem{CKL25} K. Chen, J. H. Koolen, S. Liu, Edge-connectivity of graphs with non-negative Bakry-\'{E}mery curvature and amply regular graphs, J. Lond. Math. Soc. (2) 113 (2026), no. 4, Paper No. e70527. 

\bibitem{CLY25} K. Chen, S. Liu, Z. You, Connectivity versus Lin--Lu--Yau curvature, Int. Math. Res. Not. IMRN (2025), no. 19, rnaf303. .

\bibitem{Cush20} D. Cushing, S. Kamtue, J. Koolen, S. Liu, F. M\"{u}nch, N. Peyerimhoff, Rigidity of the Bonnet-Myers inequality for graphs with respect to Ollivier Ricci curvature, Adv. Math. 369 (2020), 107188.

\bibitem{Cush22} D. Cushing, R. Kangaslampi, V. Lipi\"ainen, S. Liu, G. W. Stagg, The graph curvature calculator and the curvatures of cubic graphs, Exp. Math. 31 (2022), no. 2, 583-595.

\bibitem{Hall80} J. I. Hall, Locally Petersen graphs, J. Graph Theory 4 (1980), no. 2, 173-187.

\bibitem{HPS25} P. Horn, A. Purcilly, A. Stevens, Graph curvature and local discrepancy, J. Graph
 Theory 108 (2025), no. 2, 337-360.

\bibitem{HM25} B. Hua, F. M\"unch, Every salami has two ends,
J. Reine Angew. Math. 821 (2025), 291-321.

\bibitem{HLX24} X. Huang, S. Liu, Q. Xia, 
Bounding the diameter and eigenvalues of amply regular graphs via Lin--Lu--Yau curvature,
Combinatorica 44 (2024), no. 6, 1177-1192.

\bibitem{Jost} J. Jost, Riemannian geometry and geometric analysis, Seventh edition, Universitext, Springer, Cham, 2017. 

\bibitem{JL14} J. Jost, S. Liu,
Ollivier's Ricci curvature, local clustering and curvature-dimension inequalities on graphs,
Discrete Comput. Geom. 51 (2014), no. 2, 300-322.

\bibitem{LY10} Y. Lin, S.-T. Yau, 
Ricci curvature and eigenvalue estimate on locally finite graphs,
Math. Res. Lett. 17 (2010), no. 2, 343-356.

\bibitem{LLY11} Y. Lin, L. Lu, S.-T. Yau, Ricci curvature of graphs, Tohoku Math. J. 63 (2011), no. 4, 605-627.
\bibitem{LW20} L. Lu, Z. Wang, On the size of planar graphs with positive Lin--Lu--Yau Ricci curvature, arXiv: 2010.03716, 2020.
\bibitem{Mad71} W. Mader, Minimale $n$-fach kantenzusammenh\"{a}ngende Graphen, Math. Ann. 191 (1971), 21-28.
\bibitem{MW19} F. M\"unch, R. Wojciechowski, Ollivier Ricci curvature for general graph Laplacians: Heat equation, Laplacian comparison, non-explosion and diameter bounds, Adv. Math. 356 (2019), 106759.
\bibitem{NR17} L. Najman and P. Romon (Eds), Modern approaches to discrete curvature, Lecture Notes in Math., 2184, Springer, Cham, 2017.

\bibitem{Ollivier09} Y. Ollivier, Ricci curvature of Markov chains on metric spaces, J. Funct. Anal. 256 (2009), no. 3, 810-864.
\bibitem{Button3} Y. Ollivier, C. Villani, A curved Brunn-Minkowski inequality on the discrete hypercube, or: what is the Ricci curvature of the discrete hypercube, SIAM J. Discrete Math. 26 (2012), no. 3, 983-996.

\bibitem{Salez22} J. Salez, Sparse expanders have negative curvature, Geom. Funct. Anal. 32 (2022), no. 6, 1486-1513.

\bibitem{Smith14} J. D. H. Smith, Ricci curvature, circulants, and a matching condition, Discrete Math. 329 (2014), 88-98.

\bibitem{Tutte1} W. T. Tutte, The factorization of linear graphs, J. London Math. Soc. 22 (1947), 107-111.
\bibitem{Tutte2} W. T. Tutte, The factors of graphs, Canad. J. Math. 4 (1952), 314-328.

\bibitem{Villani03} C. Villani, Topics in Optimal Transportation, Grad. Stud. Math., 58, American Mathematical Society, Providence, RI, 2003.

\bibitem{Wat70} M. E. Watkins,  Connectivity of transitive graphs, J. Combinatorial Theory 8 (1970), 23-29.

\bibitem{Weetman94a} G. M. Weetman, A construction of locally homogeneous graphs, J. London Math. Soc. (2) 50 (1994), no. 1, 68-86.

\bibitem{Weetman94b} G. M. Weetman, Diameter bounds for graph extensions, J. London Math. Soc. (2) 50 (1994), no. 2, 209-221.

\end{thebibliography}
\end{document}